\documentclass[11pt]{amsart}
\usepackage{amssymb,mathrsfs,enumerate}
\usepackage{amsmath,amsfonts,amssymb,amscd,amsthm,bbm}
\usepackage{kotex}
\usepackage{caption, subcaption}
\usepackage{ulem}
\usepackage{comment}
\usepackage{graphicx} 
\usepackage{mathtools}
\usepackage{multirow}

\usepackage[table]{xcolor} 
\usepackage{algorithm}
\usepackage{algpseudocode}
\usepackage{hyperref}
\usepackage{tikz}
\hypersetup{colorlinks,breaklinks,
	linkcolor=blue,urlcolor=blue,
	anchorcolor=blue,citecolor=blue}
\usepackage[nameinlink, capitalise, noabbrev]{cleveref}
\usepackage{overpic}
\usepackage{enumitem}

\newcommand{\T}{\rule{0pt}{2.6ex}}
\newcommand{\B}{\rule[-1.2ex]{0pt}{0pt}}
\newcommand{\stkout}[1]{\ifmmode\text{\sout{\ensuremath{#1}}}\else\sout{#1}\fi}
\DeclareMathOperator*{\essinf}{ess\,inf}
\DeclareMathOperator*{\argmax}{arg\,max}

\definecolor{CBO}{RGB}{252,231,231}
\definecolor{adCBO}{RGB}{215,235,255}
\definecolor{adamCBO}{RGB}{221,247,221}

\title[Consensus based optimization with average drift]{CBO algorithm with average drift and applications to portfolio optimization}

\author[Bae]{Hyeong-Ohk Bae}
\address[Bae]{Department of Financial engineering, Ajou University, Suwon 16499, Republic of Korea}
\email{hobae@ajou.ac.kr}

\author[Ha]{Seung-Yeal Ha}
\address[Ha]{Department of Mathematical Sciences and Research Institute of Mathematics, Seoul National University, Seoul 08826, Republic of Korea}
\email{syha@snu.ac.kr}

\author[Min]{Chanho Min}
\address[Min]{Department of Financial engineering, Ajou University, Suwon 16499, Republic of Korea}
\email{chanhomin@ajou.ac.kr}

\author[Yoo]{Jane Yoo}
\address[Yoo]{Department of Financial engineering, Ajou University, Suwon 16499, Republic of Korea}
\email{janeyoo@ajou.ac.kr}

\author[Yoon]{Jaeyoung Yoon*}
\address[Yoon]{\color{black}Department of Mathematics, School of Computation Information and
	Technology,\newline Technical University of Munich, Garching bei M\"unchen 85748, Germany
}
\email{{\color{black}wodud1516@gmail.com, jaeyoung.yoon@tum.de}}

\numberwithin{equation}{section}
\newtheorem{theorem}{Theorem}[section]

\newtheorem{lemma}[theorem]{Lemma}
\newtheorem{corollary}[theorem]{Corollary}

\newtheorem{remark}[theorem]{Remark}

\newcommand{\bbr}{\mathbb R}

\newcommand{\bbe}{\mathbb E}

\newcommand{\e}{\varepsilon}

\newcommand{\br}{\bf r}

\newcommand{\bx}{\boldsymbol{x}}
\newcommand{\by}{\mbox{\boldmath $y$}}

\newcommand{\bv}{\mbox{\boldmath $v$}}

\newcommand{\bm}{\mbox{\boldmath $m$}}
\newcommand{\bg}{\mbox{\boldmath $g$}}

\begin{document}
	\date{\today}
	
	
	\thanks{\textbf{Acknowledgment.} H.O. Bae is supported by the Basic Research Program through the National Research Foundation of Korea(NRF) funded by the Ministry of Education and Technology ({\color{black}NRF-2021R1A2C1093383}) and by Ajou University Research Fund. The work of S.-Y. Ha is supported by National Research Foundation (NRF) grant funded by the Korea government(MIST) (RS-2025-00514472). J. Yoo is supported by {\color{black}NRF-S2025A040300202} and Ajou University Research Fund. {
			\color{black}J. Yoon acknowledges the Alexander von Humboldt Stiftung for support through a postdoctoral research fellowship.}}
	
	\begin{abstract}
		We propose a consensus based optimization algorithm with average drift (in short Ad-CBO) and provide a theoretical framework for it. In the theoretical analysis, we show that particle solutions to Ad-CBO converge to a global minimizer. In numerical simulations, we examine Ad-CBO's performance in optimizing static and dynamic objective functions. As a real-time application, we test the efficiency of Ad-CBO to find the optimal portfolio given stochastically evolving multi-asset prices in a financial market. The proposed Ad-CBO  exhibits higher searching speed, lower tracking errors and regret bound than the CBO without stochastic diffusion.
	\end{abstract}
	\maketitle \centerline{\date}
	

	\section{Introduction}\label{sec1}
	\setcounter{equation}{0}
	Many real-world problems are dynamic (or time-dependent) and high-dimensional. Moreover, these require a systemic approach to real-time optimizations in continuously changing environments. For intelligent robotic systems, tracking a target or planning a path in complex and dynamic environments is an essential task. Similar problems can be found in various disciplines, e.g.,  a predator chasing a flock of running prey, mutation problems in various probabilistic models of genome evolution for the distribution of gene families and online optimization problems searching for an optimal investment strategy in a large and complicated space in finance, etc.
	
	In this paper, we propose a {\it consensus based optimization algorithm with adaptive average drift} (in short, Ad-CBO) to track a moving optimum in a time-varying and complex environment. A consensus based optimization (CBO) method is a first-order particle-based optimization algorithm {\color{black}\cite{P-T-T-M}} which is fast and derivative free, unlike gradient-based algorithms \cite{Be}. Thus, it can be easily implemented to solve high-dimensional non-convex, {\color{black}non-differentiable} optimization problems \cite{C-J-L-Z, T-P-B-S}, as well as problems with constraints \cite{BHKLMY}. For the convergence and error analysis to CBO algorithms, we refer to \cite{C-C-T-T,F-K-R,H-J-Kc,HJK}. To set up the stage, we first discuss the concept of an average drift to the CBO algorithm, which can enhance its ability to track a moving target. By updating the overall position of a particle swarm proportional to the distance between the center of the distribution and weighted average, the average drift helps test particles prevent getting stuck in a suboptimal solution and have more chances to explore unseen spaces where new candidates are found. 
	
	The main results of this paper are two-fold. First, we provide a sufficient framework to Ad-CBO which can induce the emergent behavior, especially {\it consensus}, of particles. Using numerical simulations, we evaluate the performance of Ad-CBO on both static and dynamic optimization problems. For the static problem, we compare the performance of Ad-CBO with CBO in minimizing the Rastrigin function. Our results show that the average drift improves the performance of CBO and it allows faster convergence to the optimum with lower errors.
	
	Second, we investigate the computational efficiency of Ad-CBO in online portfolio selection (OPS) problem from finance, which is the on-time version of portfolio selection problem introduced by economist Markowitz \cite{M}. It deals with a trading strategy that sequentially allocates capital among a group of assets to maximize investment returns. The OPS involves finding optimal weights on multiple assets that have continuously and stochastically changing returns at every trading moment, making it a real-time searching problem. We perform simulations on real applications and obtain results showing that the average drift improves the performance of CBO in terms of optimized value, its variance and portfolio's volatility (see Table \ref{new_sim_port}).
	
	We also present analytical and numerical solutions for the upper bound of regret, which measures the difference in exponential growth rates of wealth based on the suggested investment strategy and the best constant strategy in hindsight. To the best of our knowledge, this is the first rigorous analysis of the upper bound of the regret bound, which is widely used to evaluate algorithm performance in OPS. Our approach provides a theoretical foundation for evaluating the performance of OPS algorithms and can be extended to other online optimization problems in finance and other fields. \newline
	
	The rest of this paper is organized as follows. In Section \ref{sec:2}, we review the  CBO algorithm and give interpretation of average drift. In Section \ref{sec:3}, we introduce Ad-CBO algorithm and study the convergence and error estimates. In Section \ref{sec:4}, we provide several numerical simulations and compare the performance of Ad-CBO with the baseline CBO. Finally, Section \ref{sec:5} is devoted to a brief summary of this paper and some remaining issues for a future work. Additionally, Section \ref{apdx} provides an upper bound estimate for regret bound.\newline
	
	\noindent {\bf Notation}:~For two matrices $A=(a_{ij})$ and $B=(b_{ij})$ of the same size, we denote their Hadamard product by $A\odot B$:
	\[ A\odot B=(a_{ij}b_{ij}). \]
	For notation simplicity, we also use the following handy notation:
	\[ [N] := \{ 1, \cdots, N \}, \quad {\bf 1} = (1, \cdots, 1) \in \bbr^d.  \]
	Let ${\bx}=(x^1,\cdots,x^d)$ be a random vector whose component $x^i$ is a random variable on the same probability space $(\Omega, {\mathcal F}, {\mathbb P})$ for all $i\in[d]$. Then, we denote the expectation and variance like functionals for ${\bx}$ by
	\[\mathbb{E}{\bx}:=\mathbb{E}[{\bx}]:=\left(\mathbb{E}[x^1],\cdots,\mathbb{E}[x^d]\right),\quad\mbox{Var}[{\bx}]:=\mathbb{E}\big[\|{\bx}\|^2\big]-\big\|\mathbb{E}[{\bx}]\big\|^2,\]
	where $\|\cdot\|$ is the Euclidean norm.
	
	\section{Preliminaries}  \label{sec:2}
	\setcounter{equation}{0}
	In this section, we briefly review the consensus based optimization algorithm and give an intuition of average drift.
	Consider an optimization problem for a continuous objective function $L$ over the search space $\bbr^d$:
	\[ \mbox{minimize}~~L({\bx})~\mbox{over $\bbr^d$}. \]
	In other words, we look for a global minimum point ${\bx}_\sharp$ for $L$:
	\begin{align} \label{B-1}
		\displaystyle {\bx}_\sharp \in \mbox{argmin}_{{\bx} \in \bbr^d} L({\bx}).
	\end{align}
	The CBO algorithm \cite{P-T-T-M}  is a gradient free stochastic optimization algorithm for \eqref{B-1}, and  it contains two key features, stochastic aggregation and stochastic annealing \cite{K-G-V} from statistical physics. It is first proposed in deterministic setting \cite{AS-J}, and then the stochastic noise term is added later \cite{P-T-T-M} to avoid the confinement at saddle points of the object function to be minimized in a search space. More precisely, let ${\bx}_t^k \in \bbr^d$ be the spatial position of the $k$-th test particle {\color{black} at time $t$}, and we set
	\begin{align} \label{B-2}
		{\bx}_t^k = (x_t^{k,1}, \cdots, x_t^{k,d}), \qquad   {\bf M}_{\beta,t} =(M_{\beta,t}^1, \cdots, M_{\beta,t}^d)
		:= \frac{\sum_{i=1}^{N}e^{-\beta L({\bx}^i_t)}  {\bx}^i_t}{\sum_{i=1}^{N} e^{-\beta L({\bx}^i_t)}},
	\end{align}
	where $\beta$ is a positive constant. The spatial process ${\bx}_t^i$ is assumed to be governed by the system of the first-order stochastic ordinary differential equations:
	\begin{align} \label{B-3}
		d{\bx}^i_t = -\lambda ({\bx}^i_t - {\bf M}_{\beta,t})dt
		- \sigma({\bx}_t^i- {\bf M}_{\beta,t})\odot d{\bf W}_t^i, \quad t > 0, \quad  i \in [N],
	\end{align}
	where $M_{\beta,t}$ is the weighted average in \eqref{B-2}, $\lambda$ and $\sigma$ denote the drift rate and noise intensity, respectively. Here, ${\bf W}_t^i$$:=[W_t^{i,1},\cdots,W_t^{i,d}]^\top$ is a $d$-dimensional Brownian motion {\color{black}whose elements are defined on a probability space $(\Omega,\mathcal F,\{\mathcal F_t\}_{t\ge0},\mathbb P)$}, that is $\{W_t^{i,l}\}_{l=1}^d$ are independent one-dimensional Brownian motions, which satisfy mean zero and covariance relations: for any $i\in[N]$,
	\[ \bbe[W_t^{i,l}]= 0 \quad \mbox{for $l = 1, \cdots, d$} \quad  \mbox{and} \quad  \bbe[W_t^{i,l_1} W_t^{i,l_2}] = \delta_{l_1 l_2} t, \quad 1 \leq  l_1, l_2 \leq d. \]
	The stochastic CBO algorithm \eqref{B-3} asserts that under suitable framework, thanks to Laplace principle,
	\[  {\bx}_t^i \to  {\bx}_\sharp  \in \mbox{argmin}_{{\bx} \in \bbr^d} L({\bx}), \quad \mbox{as}~~t\to\infty\quad\mbox{and}\quad\beta \to \infty, \quad i \in [N].  \]
	The emergence of a global consensus to \eqref{B-3} is also provided in \cite{C-J-L-Z}, and \cite{HJK} for the case of ${\bf W}_t^i={\bf W}_t$. The latter one \cite{HJK}, i.e., under the identical noise, contains the rigorous proof of the asymptotic convergence to a global optimal solution in continuous and discrete spaces, respectively. Besides, many CBO algorithms were introduced: CBO with adaptive momentum in \cite{CJL}
	and the two-step CBO in \cite{BHKLMY}.
	CBO and its applications have demonstrated notable computational efficiency, particularly in problems arising from particle physics data, where traditional optimization methods encounter difficulties \cite{C-J-L-Z} .\newline
	
	{\color{black}In this paper, we propose the Ad-CBO algorithm, an extension of the orginal discrete stochastic CBO model, by introducing an additional average drift term.} The main limitation of CBO model without stochastic diffusion lies in the detected region by particles is contained in that of the previous time, i.e., for a solution $\{{\bx}_t^i\}_{i=1}^N$ to \eqref{B-3} with $\sigma=0$,
	\begin{align}\label{B-3-1}
		\mbox{Conv}\left(\left\{{\bx}_t^i\right\}_{i=1}^N\right)\subseteq\mbox{Conv}\left(\left\{{\bx}_s^i\right\}_{i=1}^N\right),\quad\forall~t\ge s.
	\end{align}
	This is because the attraction term in CBO model forces all particles to move toward the weighted average of the swarm, which lies in the convex hull generated by their locations. This implies the lack of opportunity of success detection when
	\begin{align}\label{B-4}
		\mbox{argmin}_{{\bx}\in\mathbb R^d}L({\bx})\notin\mbox{Conv}\left(\left\{{\bx}_0^i\right\}_{i=1}^N\right).
	\end{align}
	{\color{black}To address this issue, the stochastic CBO model \eqref{B-3} introduces a noise term $(\sigma>0)$ that allows particles to escape the convex hull and explore previously unreachable regions. However, the noise-induced exploration is undirected and may hinder convergence. If the noise level is too small, exploration remains limited. On the other hand, if it is too large, particles may fail to reach consensus. This may eventually lead to optimization failure.
		
		To improve exploration while preserving convergence, we introduce an average drift term into the dynamics:
		\begin{align*}
			d{\bx}^i_t = -\lambda_0({\bx}^i_t - {\bf M}_{\beta,t})dt-\lambda_1\left(\bar{\bx}_t-{\bf M}_{\beta,t}\right)dt- \sigma({\bx}_t^i- {\bf M}_{\beta,t})\odot d{\bf W}_t^i, \quad t > 0, \quad  i \in [N],
		\end{align*}
		where $\bar\bx_t$ is the ensemble mean of particles and ${\bf M}_{\beta,t}$ is the weighted average.
		
		The average drift term, $-\lambda_1(\bar\bx_t-{\bf M}_{\beta,t})$, uniformly shifts all particles in the same direction, pushing the entire swarm toward the weighted average. This acts as a global guidance mechanism, helping the ensemble to relocate toward more promising regions of the search space. Crucially, unlike random noise, the average drift is deterministic and directional, and since it applies equally to all particles, it does not disrupt consensus formation.
		
		The average drift is particularly beneficial when the optimal point lies outside the initial convex hull. It effectively moves the whole particle distribution toward the boundary, facilitating the detection of new regions beyond the initially explored domain. Moreover, because the drift magnitude depends on the distance between center of mass and the weighed average, the mechanism adapts dynamically. Once a particle identifies a promising area, the average drift helps guide the entire swarm toward it, promoting more detailed and focused exploration.
	}

	\section{Discrete CBO with Adaptive Average Drift and noise (Ad-CBO)}
	\label{sec:3}
	\setcounter{equation}{0}
	In this section, we propose the discrete CBO algorithm with adaptive average drift (Ad-CBO) and discuss its emergent behavior in consensus and pursue the optimal solution. For a given Lipschitz continuous objective function $L$, we look for a global optimal point ${\bx}_\sharp = (x_\sharp^1, \cdots, x_\sharp^d)  \in \bbr^d$ among sample points ${\bx}_n^i = (x_n^{i,1}, \cdots, x_n^{i,d}) \in \bbr^d$, which is the $i$-th sample point ($i\in[N]$) at time step $n$.
	
	First, we introduce the discrete Ad-CBO with noise. For a given time step $h$, let
	\[{\bf W}_{h,n}=(W_{h,n}^1,\cdots,W_{h,n}^d)\in\mathbb R^d,\quad\forall ~n\in\mathbb N\cup\{0\},\]
	be i.i.d random vectors whose elements $W_{h,n}^l$ are i.i.d random variables, defined on the common probability space $(\Omega,\mathcal F,\mathbb P)$, following
	\[W_{h,n}^l\sim\mathcal N(0,h),\quad\forall~n\in\mathbb N\cup\{0\},\quad\forall~l=1,\cdots,d.\]
	We denote the expectation in the given probability space by $\mathbb E$. Then, the proposed discrete stochastic Ad-CBO model is given as follows:
	\begin{align}\label{discrete}
		{\bx}_{n+1}^i={\bx}_n^i-\lambda_0h({\bx}_n^i-{\bf M}_{\beta,n})-\lambda_1h(\bar{\bx}_n-{\bf M}_{\beta,n})-\sigma({\bx}_n^i- {\bf M}_{\beta,n})\odot{\bf W}_{h,n},\quad n\ge0, \quad i \in[N].
	\end{align}
	We adopt the indistinguishable noise for each particle to be consistent with \cite{HJK}. Note that ${\bx}_k$ for $k\le n$ are independent of ${\bf W}_{h,n}$.
	
	\begin{theorem}[Emergence of a global consensus]\label{T3.3}
		For a solution process $\{{\bx}_n^i\}_{i=1}^N$ to \eqref{discrete}, we have
		\begin{align*}
			{
				\mathbb{E}\Big[\|{\bx}_n^i-{\bx}_n^j\|^2\Big]
				\le e^{-nh(2\lambda_0-\lambda_0^2h-\sigma^2)}
				\mathbb{E}\Big[\| {\bx}_0^i-{\bx}_0^j\|^2\Big],\quad n\in\mathbb{N}, \quad i, j \in [N].
			}
		\end{align*}
	\end{theorem}
	\begin{proof} It follows from \eqref{discrete} that
		\begin{align}\label{T3.3P1}
			x_{n+1}^{i,l}-x_{n+1}^{j,l}=(1-\lambda_0h-\sigma W_{h,n}^l)(x_n^{i,l}-x_n^{j,l}).
		\end{align}
		We square and take the expectation on both sides to get
		\begin{align}\label{T1.2peq}
			\begin{aligned}
				\mathbb{E}\Big[(x_{n+1}^{i,l}-x_{n+1}^{j,l})^2\Big]&=\mathbb{E}\Big[(1-\lambda_0h-\sigma W_{h,n}^l)^2(x_n^{i,l}-x_n^{j,l})^2\Big]\\
				&=\mathbb{E}\Big[(1-\lambda_0h-\sigma W_{h,n}^l)^2\Big]\mathbb{E}\Big[(x_n^{i,l}-x_n^{j,l})^2\Big]\\
				&=\Big((1-\lambda_0h)^2+\sigma^2h\Big)\mathbb{E}\Big[(x_n^{i,l}-x_n^{j,l})^2\Big],
			\end{aligned}
		\end{align}
		where we have used the independence of $ W_{h,n}^l$ and $(x_n^{i,l}-x_n^{j,l})$. Now, we sum up \eqref{T1.2peq} over $l=1,\cdots,d$ to find
		\begin{align*}
			\mathbb{E}\Big[\|{\bx}_{n+1}^i-{\bx}_{n+1}^j\|^2\Big]=\Big(1-h(2\lambda_0-\lambda_0^2h-\sigma^2)\Big)\mathbb{E}\Big[\|{\bx}_n^i-{\bx}_n^j\|^2\Big].
		\end{align*}
		Next, we use the inequality $1+x\le e^x$ to obtain
		\begin{align*}
			\mathbb{E}\Big[\|{\bx}_{n+1}^i-{\bx}_{n+1}^j\|^2\Big]\le e^{-h(2\lambda_0-\lambda_0^2h-\sigma^2)}\mathbb{E}\Big[\|{\bx}_n^i-{\bx}_n^j\|^2\Big].
		\end{align*}
		This yields
		\[  \mathbb{E}\Big[\|{\bx}_n^i-{\bx}_n^j\|^2\Big]
		\le e^{-nh(2\lambda_0-\lambda_0^2h-\sigma^2)}
		\mathbb{E}\Big[\|{\bx}_0^i-{\bx}_0^j\|^2\Big],\quad n\in\mathbb{N}. \]
	\end{proof}
	
	\begin{lemma}\label{lem_ele}
		For a solution process $\left\{{\bx}_n^i\right\}_{i=1}^N$ to \eqref{discrete}, we have the following two assertions:
		\begin{align*}
			&(i)~x_n^{i,l}-\bar{x}_n^l=\prod_{k=0}^{n-1}\left(1-\lambda_0h-\sigma W_{h,k}^l\right)(x_0^{i,l}-\bar{x}_0^l),\quad\forall~l=1,\cdots,d,~~\forall~i\in[N].\\
			&(ii)~\max_{i,j\in[N]}(x_n^{i,l}-x_n^{j,l})=\prod_{k=0}^{n-1}\left(1-\lambda_0h-\sigma W_{h,k}^l\right)\max_{i,j\in[N]}(x_0^{i,l}-x_0^{j,l}),\quad\forall~l=1,\cdots,d.
		\end{align*}
	\end{lemma}
	\begin{proof}
		(i)~It follows from \eqref{discrete} that
		\begin{align*}
			\bar{\bx}_{n+1}&=\bar{\bx}_n-\lambda_0h(\bar{\bx}_n-{\bf M}_{\beta,n})-\lambda_1h(\bar{\bx}_n-{\bf M}_{\beta,n})-\sigma(\bar{\bx}_n-{\bf M}_{\beta,n})\odot{\bf W}_{h,n}.
		\end{align*}
		We subtract the above relation from \eqref{discrete} and consider $l$th element ($l=1,\cdots,d$),
		\begin{align*}
			x_{n+1}^{i,l}-\bar{x}_{n+1}^l&=x_n^{i,l}-\bar{x}_n^l-\lambda_0h(x_n^{i,l}-\bar{x}_n^l)-\sigma(x_n^{i,l}-\bar{x}_n^l)W_{h,n}^l\\
			&=\left(1-\lambda_0h-\sigma W_{h,n}^l\right)(x_n^{i,l}-\bar{x}_n^l).
		\end{align*}
		Sequentially, we can get the desired estimate.\newline
		
		\noindent(ii)~From \eqref{T3.3P1}, we have
		\begin{align*}
			\max_{i,j\in[N]}(x_n^{i,l}-x_n^{j,l})&=\max_{i,j\in[N]}\left(1-\lambda_0h-\sigma W_{h,n-1}^l\right)(x_{n-1}^{i,l}-x_{n-1}^{j,l})\\
			&=\left(1-\lambda_0h-\sigma W_{h,n-1}^l\right)\max_{i,j\in[N]}(x_{n-1}^{i,l}-x_{n-1}^{j,l}).
		\end{align*}
		Again sequentially, we obtained the desired result.
	\end{proof}
	\begin{lemma}\label{lem_exp}
		Suppose that
		\begin{align}\label{cond_exp}
			0<h,\quad0<\lambda_0h<1\quad\mbox{and}\quad\frac{1-\lambda_0h}{y^*\sqrt{h}}>\sigma,
		\end{align}
		where $y^*=y^*(\lambda_0h)$ satisfies
		\[\frac{1}{y^*}\sqrt{\frac{\pi}{2}}\exp\left(-\frac{(y^*)^2}{2}\right)+\mbox{erf}\left(\frac{y^*}{\sqrt{2}}\right)-\frac{1}{1-\lambda_0h}=0.\]
		Then, for a Gaussian random variable $\eta_h\sim\mathcal N(0,h)$, we have
		\begin{align*}
			\mathbb E\big|1-\lambda_0h+\sigma\eta_h\big|<1.
		\end{align*}
	\end{lemma}
	\begin{proof}
		First, note that
		\[1-\lambda_0h+\sigma\eta_h\sim\mathcal N(1-\lambda_0h,\sigma^2h).\]
		By the folded normal distribution, we have
		
		\begin{align*}
			\mathbb E\big|1-\lambda_0h+\sigma\eta_h\big|=\sigma\sqrt{\frac{\pi h}{2}}\exp\left(-\frac{1}{2}\left(\frac{1-\lambda_0h}{\sigma\sqrt{h}}\right)^2\right)-(1-\lambda_0h)\mbox{erf}\left(-\frac{1-\lambda_0h}{\sigma\sqrt{2h}}\right),
		\end{align*}
		i.e., we need to show
		\begin{align}\label{bb1}
			\sigma\sqrt{\frac{\pi h}{2}}\exp\left(-\frac{1}{2}\left(\frac{1-\lambda_0h}{\sigma\sqrt{h}}\right)^2\right)-(1-\lambda_0h)\mbox{erf}\left(-\frac{1-\lambda_0h}{\sigma\sqrt{2h}}\right)<1.
		\end{align}
		where {\color{black}$\mbox{erf}(\cdot)$ is the Gauss error function.}
		
		{\color{black}To show \eqref{bb1},} we use the notation
		\[y=\frac{1-\lambda_0h}{\sigma\sqrt{h}}>0,\]
		to find that the relation \eqref{bb1} is equivalent to
		\begin{align}\label{bb2}
			\sqrt{\frac{\pi}{2}}\exp\left(-\frac{y^2}{2}\right)-y\mbox{erf}\left(-\frac{y}{\sqrt{2}}\right)-\frac{y}{1-\lambda_0h}<0.
		\end{align}
		Define a function $g_{\lambda_0h}:{\color{black}(0,\infty)}\to\mathbb R$ by
		\begin{align*}
			g_{\lambda_0h}(y):=\frac{1}{y}\sqrt{\frac{\pi}{2}}\exp\left(-\frac{y^2}{2}\right)-\mbox{erf}\left(-\frac{y}{\sqrt{2}}\right)-\frac{1}{1-\lambda_0h},
		\end{align*}
		to rewrite the desired condition \eqref{bb2} as
		\begin{align*}
			g_{\lambda_0h}(y)<0.
		\end{align*}
		\noindent One can easily check that $g_{\lambda_0h}$ is monotone decreasing on {\color{black}$(0,\infty)$} and
		\[\lim_{y\to0+}g_{\lambda_0h}(y)=+\infty,\quad{\color{black}\lim_{y\to\infty}g_{\lambda_0h}(y)=1-\frac{1}{1-\lambda_h}<0.}\]
		This implies that the existence of the unique zero $y^*(\lambda_0h)>0$ of $g_{\lambda_0h}$, i.e., ${\color{black}g_{\lambda_0h}}\big(y^*(\lambda_0h)\big)=0$,
		and the conditions for parameters show
		\[{\color{black}g_{\lambda_0h}}\left(\frac{1-\lambda_0h}{\sigma\sqrt{h}}\right)<{\color{black}g_{\lambda_0h}}\big(y^*(\lambda_0h)\big)=0.\]
		Therefore, the relation \eqref{bb1} holds and we get the desired estimate.
	\end{proof}
	\begin{remark}
		\begin{enumerate}
			\item By conditions in Lemma \ref{lem_exp}, 
			the point is $\lambda_0$ and $\sigma$ have a positive correlation. For a fixed time step $h>0$, as $\lambda_0$ gets {\color{black}smaller}, i.e., $1-\lambda_0h$ gets {\color{black}larger}, the zeros $y^*(\lambda_0h)$ gets {\color{black}larger} due to structure of $g_{\lambda_0h}$ in the proof. This gives a {\color{black}smaller} upper bound condition for $\sigma$ in \eqref{cond_exp}. This relation seems quite clear considering the result of Lemma \ref{lem_exp}. For a Gaussian random variable $x\sim\mathcal N(\mu,s^2)$ with $\mu>0$, the folded normal distribution, i.e. a distribution for $|x|$, {\color{black}can have a small expectation when the mean $\mu$ and standard deviation $s$ are suitably balanced: larger $\mu$ requires smaller $s$, while smaller $\mu$ allows larger $s$. This comes from the fact that $\mathbb E[|x|]$ is monotone increasing in $\mu$ and $s$, respectively.}
			\item The conditions \eqref{cond_exp} is a sharp estimate, which can be confirmed from the definition of $y^*$ and $g_{\lambda_0h}$ in the proof of Lemma \ref{lem_exp}. Here, we give an explicit example for this condition: Suppose that $\lambda_0h=0.15$. Then, we have
			\[g_{\lambda_0h}(\sqrt{2})<\frac{\sqrt{\pi}}{2e}+0.85-\frac{1}{1-\lambda_0h}<0,\]
			where we used $\mbox{erf}(1)<0.85$. {\color{black}This implies $y_*(\lambda_0h)<\sqrt{2}$. Now, we choose $\sigma$ and $h$ to satisfy
				\[\sigma\sqrt{h}<\frac{1-\lambda_0h}{\sqrt{2}}=\frac{0.85}{\sqrt{2}}\approx0.601,\]
				so that $\lambda_0$ is decided and} the condition \eqref{cond_exp} holds.
		\end{enumerate}
	\end{remark}
	\begin{lemma}\label{lem_init}
		For any set of positive numbers $\left\{s_l\right\}_{l=1}^d$, suppose that initial data $\left\{{\bx}_0^i\right\}_{i=0}^N$ are i.i.d random vectors, whose elements ${\bx}_0^i=(x_0^1,\cdots,x_0^d)$ are also i.i.d Gaussian random variables such that
		\begin{align}\label{cond_init}
			x_0^l\sim\mathcal N(0,s_l^2),\quad\forall~l=1,\cdots,d.
		\end{align}
		Then, for a solution process $\left\{{\bx}_n^i\right\}_{i=1}^N$ to \eqref{discrete}, there exists a positive constant $\tilde C=\tilde C(N,s_1,\cdots,s_d)$ such that
		\begin{align*}
			&(i)~\mathbb E\left[\max_{i\in[N]}\sum_{l=1}^d|x_0^{i,l}-\bar{x}_0^l|\right]<\tilde C.\\
			&(ii)~\sum_{l=1}^d\mathbb E\left[(x_0^{i,l}-\bar{x}_0^l)^2\right]<\tilde C,\quad\forall~i\in[N].\\
			&(iii)~\sum_{l=1}^d\mathbb E\left[\max_{i,j\in[N]}(x_0^{i,l}-x_0^{j,l})^2\right]<\tilde C.
		\end{align*}
	\end{lemma}
	\begin{proof}
		(i)~Since this estimate depends on $N$, we use the notation with $N$ as follows:
		\begin{align*}
			u_{N,i,l}:=x_0^{i,l}-\bar{x}_0^l\quad\mbox{and}\quad U_{N,i}:=\sum_{l=1}^d|u_{N,i,l}|.
		\end{align*}
		So, we have
		\begin{align*}
			\mathbb E\left[\max_{i\in[N]}\sum_{l=1}^d|x_0^{i,l}-\bar{x}_0^l|\right]=\mathbb E\left[\max_{i\in[N]}U_{N,i}\right]=\mathbb E\left[\max_{i\in[N]}\sum_{l=1}^d|u_{N,i,l}|\right].
		\end{align*}
		Note that $\left\{u_{N,i,l}\right\}_{l=1}^d$ are independent. We explore the distribution of $u_{N,i,l}$, $|u_{N,i,l}|$ and $U_{N,i}$ one by one.\newline
		
		\noindent$\bullet$ (Distribution of $u_{N,i,l}$)~By the linearity of expectation, we have
		\begin{align*}
			\mathbb E[u_{N,i,l}]&=\mathbb E\left[\frac{1}{N}\sum_{j=1}^N(x_0^{i,l}-x_0^{j,l})\right]=\frac{1}{N}\sum_{j=1}^N\left(\mathbb E[x_0^{i,l}]-\mathbb E[x_0^{j,l}]\right)=0.
		\end{align*}
		For the variance, we have
		\[x_0^{i,l}-x_0^{j,l}\sim\mathcal N(0,2s_l^2),\quad\mbox{if}~~i\ne j,\]
		and
		\[\mbox{Cov}\Big(x_0^{i,l}-x_0^{j_1,l},x_0^{i,l}-x_0^{j_2,l}\Big)=s_l^2,\quad\mbox{if}~~j_1\ne j_2\in[N]\setminus\left\{i\right\},\]
		to derive
		\begin{align*}
			\mbox{Var}(u_{N,i,l})&=\mbox{Var}\left(\frac{1}{N}\sum_{j=1}^N(x_0^{i,l}-x_0^{j,l})\right)=\frac{1}{N^2}\mbox{Var}\left(\sum_{j=1}^N(x_0^{i,l}-x_0^{j,l})\right)\\
			&=\frac{1}{N^2}\left(\sum_{j=1}^N\mbox{Var}\Big(x_0^{i,l}-x_0^{j,l}\Big)+2\sum_{1\le j_1<j_2\le N}\mbox{Cov}\Big(x_0^{i,l}-x_0^{j_1,l},x_0^{i,l}-x_0^{j_2,l}\Big)\right)\\
			&=\frac{1}{N^2}\Big(2(N-1)s_l^2+(N-2)(N-1)s_l^2\Big)=\frac{N-1}{N}s_l^2.
		\end{align*}
		Since the summation of Gaussian random variables is Gaussian, one gets
		\begin{align}\label{cc1}
			u_{N,i,l}\sim\mathcal N\left(0,\frac{N-1}{N}s_l^2\right).
		\end{align}\newline
		
		\noindent$\bullet$ (Distribution of $|u_{N,i,l}|$)~We use the folded normal distribution to have
		\begin{align}\label{cc2}
			\mathbb E\big[|u_{N,i,l}|\big]=s_l\sqrt{\frac{N-1}{N}}\sqrt{\frac{2}{\pi}},\quad\mbox{Var}\big(|u_{N,i,l}|\big)=\frac{N-1}{N}s_l^2\left(1-\frac{2}{\pi}\right).
		\end{align}\newline
		
		\noindent$\bullet$ (Distribution of $U_{N,i}$)~We use the independence of $\left\{|u_{N,i,l}|\right\}_{l=1}^d$ and \eqref{cc2} to find
		\begin{align*}
			\mathbb E[U_{N,i}]&=\mathbb E\left[\sum_{l=1}^d|u_{N,i,l}|\right]=\sqrt{\frac{N-1}{N}}\sqrt{\frac{2}{\pi}}\sum_{l=1}^ds_l,\\
			\mbox{Var}(U_{N,i})&=\mbox{Var}\left(\sum_{l=1}^d|u_{N,i,l}|\right)=\sum_{l=1}^d\mbox{Var}\big(|u_{N,i,l}|\big)=\frac{N-1}{N}\left(1-\frac{2}{\pi}\right)\sum_{l=1}^ds_l^2.
		\end{align*}
		On the other hand, for any nondecreasing convex function $\varphi:\mathbb R\to\mathbb R$ and random variables $X_i$ for $i\in[N]$, Jensen's inequality shows
		\begin{align}\label{cc3}
			\varphi\left(\mathbb E\left[\max_{i\in[N]}X_i\right]\right)&\le\mathbb E\left[\max_{i\in[N]}\varphi(X_i)\right]\le\sum_{i=1}^N\mathbb E\big[\varphi(X_i)\big].
		\end{align}
		With $\phi(x)=x^2$ and $X_i=U_{N,i}\ge0$ in \eqref{cc3}, one can get
		\begin{align*}
			\mathbb E\left[\max_{i\in[N]}U_{N,i}\right]&\le\sqrt{\sum_{i=1}^N\mathbb E\left[(U_{N,i})^2\right]}\le\sqrt{N-1}\cdot\Bigg|\sum_{l=1}^ds_l\Bigg|.
		\end{align*}\newline
		
		\noindent(ii)~We use \eqref{cc1} to find
		\begin{align*}
			\sum_{l=1}^d\mathbb E\left[(x_0^{i,l}-\bar{x}_0^l)^2\right]&=\sum_{l=1}^d\mathbb E[u_{N,i,l}^2]=\frac{N-1}{N}\sum_{l=1}^ds_l^2.
		\end{align*}\newline
		
		\noindent(iii)~Note that
		\[\mathbb E\left[\max_{i,j\in[N]}(x_0^{i,l}-x_0^{j,l})^2\right]\le\mathbb E\left[\max_{i\in[N]}(2x_0^{i,l})^2\right]=4\mathbb E\left[\max_{i\in[N]}(x_0^{i,l})^2\right].\]
		We set $\varphi(x)=x^2$ and $X_i=(x_0^{i,l})^2$ in \eqref{cc3} to derive
		\begin{align*}
			\mathbb E\left[\max_{i\in[N]}(x_0^{i,l})^2\right]\le\sqrt{\sum_{i=1}^N\mathbb E\left[(x_0^{i,l})^4\right]}=\sqrt{3Ns_l^4}=s_l^2\sqrt{3N}.
		\end{align*}
		Hence, one gets
		\begin{align*}
			\sum_{l=1}^d\mathbb E\left[\max_{i,j\in[N]}(x_0^{i,l}-x_0^{j,l})^2\right]&\le4\sqrt{3N}\sum_{l=1}^ds_l^2.
		\end{align*}\newline
		
		\noindent Now, we set
		\[\tilde C(N,s_1,\cdots,s_d):=4\sqrt{3N}\max\left\{\Bigg|\sum_{l=1}^ds_l\Bigg|,~\sum_{l=1}^ds_l^2\right\},\]
		we get the desired constant.
	\end{proof}
	For now on, we set
	\[A_n:=\sum_{p=0}^n\frac{1}{N}\sum_{i=1}^N\|{\bx}_p^i-{\bf M}_{\beta,p}\|,\quad B_n:=\sum_{p=0}^n\frac{1}{N}\sum_{i=1}^N\|({\bx}_p^i-{\bf M}_{\beta,p})\odot{\bf W}_{h,p}\|.\]
	\begin{corollary}\label{cor_AB}
		Suppose that system parameters satisfy \eqref{cond_exp}, \eqref{cond_init} and
		\begin{align}\label{cond_para}
			\Big((1-\lambda_0h)^2+\sigma^2\Big)<1.
		\end{align}
		Then, for a solution process $\left\{{\bx}_n^i\right\}_{i=1}^N$, there exists a positive constant
		\[C=C(N,s_1,\cdots,s_d,h,\lambda_0,h,\sigma)\]
		such that
		\[\mathbb E\left[A_n\right]\le C\quad\mbox{and}\quad\mathbb E\left[B_n\right]\le \sqrt{h}C,\quad\forall~n\in\mathbb N.\]
	\end{corollary}
	
	\begin{proof}
		We split the proof in two parts for $A_n$ and $B_n$, respectively.\newline
		
		\noindent(i)~First, note that
		\begin{align}\label{aa1}
			\|{\bx}_p^i-{\bf M}_{\beta,p}\|\le\|{\bx}_p^i-\bar{\bx}_p\|+\|\bar{\bx}_p-{\bf M}_{\beta,p}\|\le2\max_{i\in[N]}\|{\bx}_p^i-\bar{\bx}_p\|,
		\end{align}
		where we used the fact that $\bar{\bx}_p$ and ${\bf M}_{\beta,p}$ are in the convex hull of $\left\{{\bx}_p^i\right\}_{i=1}^N$.
		
		On the other hand, from Lemma \ref{lem_ele} (i), one can derive
		\begin{align}\label{aa2}
			\begin{aligned}
				\|{\bx}_p^i-\bar{\bx}_p\|&=\sqrt{\sum_{l=1}^d(x_0^{i,l}-\bar{x}_0^l)^2\prod_{k=0}^{p-1}\left(1-\lambda_0h-\sigma W_{h,k}^l\right)^2}\\
				&\le\sum_{l=1}^d\big|x_0^{i,l}-\bar{x}_0^l\big|\prod_{k=0}^{p-1}\big|1-\lambda_0h-\sigma W_{h,k}^l\big|.
			\end{aligned}
		\end{align}
		We put \eqref{aa2} to \eqref{aa1} to get
		\begin{align*}
			\|{\bx}_p^i-{\bf M}_{\beta,p}\|\le2\max_{i\in[N]}\sum_{l=1}^d\big|x_0^{i,l}-\bar{x}_0^l\big|\prod_{k=0}^{p-1}\big|1-\lambda_0h-\sigma W_{h,k}^l\big|.
		\end{align*}
		We combine the above relation and Lemma \ref{lem_exp} with the notation
		\[\alpha(\lambda_0,h,\sigma):=\mathbb E\big|1-\lambda_0h-\sigma\eta_h\big|,\quad\forall~\eta_h\sim\mathcal N(0,h),\]
		to find
		\begin{align}\label{aa2-1-0}
			\begin{aligned}
				\mathbb E\|{\bx}_p^i-{\bf M}_{\beta,p}\|&\le2\mathbb E\left[\max_{i\in[N]}\sum_{l=1}^d|x_0^{i,l}-\bar{x}_0^l|\right]\prod_{k=0}^{p-1}\mathbb E\big|1-\lambda_0h-\sigma W_{h,k}^l\big|\\
				&\le2\mathbb E\left[\max_{i\in[N]}\sum_{l=1}^d|x_0^{i,l}-\bar{x}_0^l|\right]\alpha^p\le2\tilde C\alpha^p,
			\end{aligned}
		\end{align}
		where the constant $\tilde C=\tilde C(N,s_1,\cdots,s_d)$ is the same constant in Lemma \ref{lem_init}. This leads to
		\begin{align*}
			\mathbb E\left[A_n\right]=\sum_{p=0}^n\frac{1}{N}\sum_{i=1}^N\mathbb E\|{\bx}_p^i-{\bf M}_{\beta,p}\|\le2\tilde C\sum_{p=0}^n\alpha^p\le\frac{2\tilde {C}}{1-\alpha}.
		\end{align*}\newline
		
		\noindent(ii)~We split the term in $B_n$ as
		\begin{align}\label{aa2-1}
			\begin{aligned}
				\|({\bx}_n^i-{\bf M}_{\beta,n})\odot{\bf W}_{h,n}\|&\le\|({\bx}_n^i-\bar{\bx}_n)\odot{\bf W}_{h,n}\|+\|(\bar{\bx}_n-{\bf M}_{\beta,n})\odot{\bf W}_{h,n}\|\\
				&=:\mathcal I_1+\mathcal I_2,
			\end{aligned}
		\end{align}
		and make estimates for $\mathcal I_1$ and $\mathcal I_2$, separately.\newline
		
		\noindent$\diamond$ (Estimate of $\mathbb E\mathcal I_1$)~We use Jensen's inequality and Lemma \ref{lem_ele} (i) to get
		\begin{align*}
			&\Big(\mathbb E\|({\bx}_n^i-\bar{\bx}_n)\odot{\bf W}_{h,n}\|\Big)^2\le\mathbb E\Big[\|({\bx}_n^i-\bar{\bx}_n)\odot{\bf W}_{h,n}\|^2\Big]\\
			&\hspace{1cm}=\mathbb E\left[\sum_{l=1}^d(x_0^{i,l}-\bar{x}_0^l)^2\prod_{k=0}^{n-1}\left(1-\lambda_0h-\sigma W_{h,k}^l\right)^2(W_{h,n}^k)^2\right]\\
			&\hspace{1cm}=\sum_{l=1}^d\mathbb E\Big[(x_0^{i,l}-\bar{x}_0^l)^2\Big]\prod_{k=0}^{n-1}\mathbb E\left[\left(1-\lambda_0h-\sigma W_{h,k}^l\right)^2\right]\mathbb E\Big[(W_{h,n}^k)^2\Big]\\
			&\hspace{1cm}=h\Big((1-\lambda_0h)^2+\sigma^2\Big)^n\sum_{l=1}^d\mathbb E\Big[(x_0^{i,l}-\bar{x}_0^l)^2\Big].
		\end{align*}
		So, one can have
		\begin{align}\label{aa3}
			\mathbb E\|({\bx}_n^i-\bar{\bx}_n)\odot{\bf W}_{h,n}\|\le\Big((1-\lambda_0h)^2+\sigma^2\Big)^{n/2}\sqrt{h\sum_{l=1}^d\mathbb E\Big[(x_0^{i,l}-\bar{x}_0^l)^2\Big]}.
		\end{align}\newline
		
		\noindent$\diamond$ (Estimate of $\mathbb E\mathcal I_2$)~We use Jensen's inequality to get
		\begin{align}\label{aa4}
			\begin{aligned}
				&\Big(\mathbb E\|(\bar{\bx}_n-{\bf M}_{\beta,n})\odot{\bf W}_{h,n}\|\Big)^2\\
				&\hspace{1cm}\le\mathbb E\Big[\|(\bar{\bx}_n-{\bf M}_{\beta,n})\odot{\bf W}_{h,n}\|^2\Big]=\mathbb E\left[\sum_{l=1}^d(\bar{x}_n^l-M_{\beta,n}^l)^2(W_{h,n}^l)^2\right]\\
				&\hspace{1cm}=\sum_{l=1}^d\mathbb E\left[(\bar{x}_n^l-M_{\beta,n}^l)^2\right]\mathbb E\left[(W_{h,n}^l)^2\right]=h\sum_{l=1}^d\mathbb E\left[(\bar{x}_n^l-M_{\beta,n}^l)^2\right].
			\end{aligned}
		\end{align}
		We again use the fact that $\bar{\bx}_n,{\bf M}_{\beta,n}\in\mbox{Conv}\left(\left\{{\bx}_n^i\right\}_{i=1}^N\right)$ to see
		\[|\bar{x}_n^l-M_{\beta,n}^l|\le\max_{i,j\in[N]}|x_n^{i,l}-x_n^{j,l}|.\]
		Then, we combine this relation and Lemma \ref{lem_ele} (ii) to find
		\begin{align}\label{aa5}
			\begin{aligned}
				\mathbb E\left[(\bar{x}_n^l-M_{\beta,n}^l)^2\right]&\le\mathbb E\left[\max_{i,j\in[N]}(x_n^{i,l}-x_n^{j,l})^2\right]\\
				&=\mathbb E\left[\prod_{k=0}^{n-1}\left(1-\lambda_0h-\sigma W_{h,k}^l\right)^2\max_{i,j\in[N]}(x_0^{i,l}-x_0^{j,l})^2\right]\\
				&=\Big((1-\lambda_0h)^2+\sigma^2\Big)^n\mathbb E\left[\max_{i,j\in[N]}(x_0^{i,l}-x_0^{j,l})^2\right].
			\end{aligned}
		\end{align}
		Now, we put \eqref{aa5} to \eqref{aa4} to derive
		\begin{align}\label{aa6}
			\mathbb E\|(\bar{\bx}_n-{\bf M}_{\beta,n})\odot{\bf W}_{h,n}\|\le\Big((1-\lambda_0h)^2+\sigma^2\Big)^{n/2}\sqrt{h\sum_{l=1}^d\mathbb E\left[\max_{i,j\in[N]}(x_0^{i,l}-x_0^{j,l})^2\right]}.
		\end{align}\newline
		
		\noindent Finally, we combine \eqref{aa2-1}, \eqref{aa3} and \eqref{aa6} to get
		\begin{align}\label{aa7}
			\begin{aligned}
				&\mathbb E\|({\bx}_n^i-{\bf M}_{\beta,n})\odot{\bf W}_{h,n}\|\\
				&\hspace{1cm}\le\Big((1-\lambda_0h)^2+\sigma^2\Big)^{n/2}\left(\sqrt{h\sum_{l=1}^d\mathbb E\Big[(x_0^{i,l}-\bar{x}_0^l)^2\Big]}+\sqrt{h\sum_{l=1}^d\mathbb E\left[\max_{i,j\in[N]}(x_0^{i,l}-x_0^{j,l})^2\right]}\right)\\
				&\hspace{1cm}\le2\sqrt{h\tilde C}\Big((1-\lambda_0h)^2+\sigma^2\Big)^{n/2},
			\end{aligned}
		\end{align}
		where we used Lemma \ref{lem_init}. Hence, one can derive
		\begin{align*}
			\mathbb E\left[B_n\right]&=\sum_{p=0}^n\frac{1}{N}\sum_{i=1}^N\mathbb E\|({\bx}_p^i-{\bf M}_{\beta,p})\odot{\bf W}_{h,p}\|\\
			&\le2\sqrt{h\tilde C}\sum_{p=0}^n\Big((1-\lambda_0h)^2+\sigma^2\Big)^{n/2}<\frac{2\sqrt{h\tilde C}}{1-\sqrt{(1-\lambda_0h)^2+\sigma^2}}.
		\end{align*}\newline
		
		\noindent Now, the constant $C$ determined as
		\[C:=\max\left\{\frac{2\tilde C}{1-\alpha},~\frac{2\sqrt{\tilde C}}{1-\sqrt{(1-\lambda_0h)^2+\sigma^2}}\right\}\]
		satisfies the desire estimates.
	\end{proof}
	In the next two theorems, we present a formation of a global consensus and convergence of consensus state to the global minimum of an object function.
	\begin{theorem}  \label{T3.6}
		Suppose system parameters and initial data satisfy \eqref{cond_exp}, \eqref{cond_init} and \eqref{cond_para}. Then, for a solution process $\{{\bx}_n^i\}_{i=1}^N$ to \eqref{discrete}, there exists a random vector ${\bx}_\infty\in\mathbb R^d$ such that
		\[\lim_{n\to\infty}{\bx}_n^i={\bx}_\infty\quad\mbox{almost surely},\quad i\in[N].\]
	\end{theorem}
	\begin{proof} We split its proof into two steps. \newline
		
		\noindent $\bullet$~Step A (${\bx}_n^i$ is a Cauchy for a.e. $\omega \in \Omega$):~First, note that
		\begin{align}\label{T3.6P1}
			\begin{aligned}
				& \|{\bx}_m^i-{\bx}_n^i\| \le\sum_{p=n}^{m-1}\|{\bx}_{p+1}^i-{\bx}_p^i\|\\
				&\hspace{0.5cm} =\sum_{p=n}^{m-1}\|-\lambda_0h({\bx}_p^i-{\bf M}_{\beta,p})-\lambda_1h(\bar{\bx}_p-{\bf M}_{\beta,p})-\sigma({\bx}_p^i-{\bf M}_{\beta,p})\odot {\bf W}_{h,p}\|\\
				&\hspace{0.5cm} \le\lambda_0h\sum_{p=n}^{m-1}\|{\bx}_p^i-{\bf M}_{\beta,p}\|+\lambda_1h\sum_{p=n}^{m-1}\|\bar{\bx}_p-{\bf M}_{\beta,p}\|+\sigma\sum_{p=n}^{m-1}\|({\bx}_p^i-{\bf M}_{\beta,p})\odot {\bf W}_{h,p}\|\\
				& \hspace{0.5cm} \le\lambda_0h(A_{m-1}-A_{n-1})+\frac{\lambda_1h}{N}\sum_{j=1}^N\left(A_{m-1}-A_{n-1}\right)+\sigma(B_{m-1}-B_{n-1}),
			\end{aligned}
		\end{align}
		where we used
		\begin{align}\label{T3.6P2}
			\|\bar{\bx}_p-{\bf M}_{\beta,p}\|=\Big\|\frac{1}{N}\sum_{j=1}^N({\bx}_p^j-{\bf M}_{\beta,p})\Big\|\le\frac{1}{N}\sum_{j=1}^N\|{\bx}_p^j-{\bf M}_{\beta,p}\|.
		\end{align}
		On the other hand, by Corollary \ref{cor_AB},  $\{ A_{n} \}_{n\in\mathbb N}$ and $\{ B_{n} \}_{n\in\mathbb N}$ are uniformly bounded sub-martingales. Hence, by Doob's Martingale convergence theorem, there exist two positive constants $A_\infty,B_\infty<\infty$ such that
		\[A_n\to A_\infty,
		\quad B_n\to B_\infty\quad\mbox{a.e.}~~\omega\in\Omega,\quad\mbox{as}~~n\to\infty,\]
		which leads to
		\[ \{ A_n \}_{n\in\mathbb N}~~ \mbox{and}~~\{ B_n \}_{n\in\mathbb N}~~\mbox{are Cauchy sequences for a.s. }\omega\in\Omega.\]
		Therefore, for a given $\e >0$, we can choose a sufficiently large $N$ such that
		\begin{align}\label{T1.3peq2}
			m,n>N\quad \Longrightarrow\quad\max \Big \{ |A_m-A_n|,|B_m-B_n|  \Big \}\le\frac{\e}{2} \quad \mbox{a.s.}~~\omega\in\Omega.
		\end{align}
		Then, \eqref{T3.6P1} and \eqref{T1.3peq2} imply that $\{{\bx}_n^i\}_{i=1}^N$ is a Cauchy. Therefore, there exists ${\bx}_\infty^i$ such that $\lim_{n\to\infty} {\bx}_n^i={\bx}_\infty^i$ a.s. \newline
		
		\noindent $\bullet$~Step B (the limit ${\bx}_\infty^i$ is independent of $i \in [N]$):~we use the global consensus result to show that ${\bx}_n^i$ converges to the same ${\bx}_\infty$  a.s.~for all $i \in [N]$.
	\end{proof}
	In the following theorem, we provide an error estimate.
	\begin{theorem} \label{T3.7}
		Suppose that an objective function, system parameter and initial data satisfy the following two conditions:
		\begin{enumerate}[label=(\roman*)]
			\item
			The objective function $L=L({\bx})$ is bounded
			with a Lipschitz constant $C_L$ and a lower bound $L_\sharp$, i.e.,
			\[|L({\bx})-L(\by)|\le C_L\|{\bx}-\by\|,\quad\forall~ {\bx},\by\in\mathbb{R}^d\quad\mbox{and}\quad\inf_{{\bx}\in\mathbb{R}^d}L({\bx})=L_\sharp.\]
			\item
			System parameters and initial data satisfy the conditions \eqref{cond_exp}, \eqref{cond_init}, \eqref{cond_para} and
			\[
			\mathbb E\left[e^{-\beta L({\bx}_0)}\right]-\beta C_Le^{-\beta L_\sharp}\left(\frac{2(\lambda_0+\lambda_1)h\tilde C}{1-\alpha}+\frac{2\sigma\sqrt{h\tilde C}}{1-\sqrt{(1-\lambda_0h)^2+\sigma^2}}\right)\ge\varepsilon\mathbb E\left[e^{-\beta L({\bx}_0)}\right]
			\]
			for some $\e\in(0,1)$ and a random variable ${\bx}_0:\Omega\to\mathbb R^d$ satisfying \eqref{cond_init}.
		\end{enumerate}
		Then, for a solution process to \eqref{discrete}, there exists a function $E: [0, \infty) \to\mathbb{R}$ such that
		\[\lim_{\beta\to\infty}E(\beta)=0 \quad \mbox{and} \quad \essinf_{\omega\in\Omega}L({\bx}_\infty(\omega))\le\essinf_{\omega\in\Omega}L({\bx}_0(\omega))+E(\beta).\]
	\end{theorem}
	\begin{proof}
		We slightly improve arguments employed in the proof of \cite[Theorem 3.2]{HJK}
		by replacing $\mathcal{C}^2$-regularity assumption of the objective function $L$ with Lipschtiz continuity. More precisely, we proceed our proof by avoiding estimates for $\nabla^2L$. Since overall proof can be found in \cite{HJK}, we rather focus on the key part: since $e^x-1\ge x$ for all $x\in\mathbb{R}$, we have
		\begin{align}\label{T1.4peq1}
			\begin{aligned}
				&\frac{1}{N}\sum_{i=1}^Ne^{-\beta L({\bx}_{n+1}^i)}-\frac{1}{N}\sum_{i=1}^Ne^{-\beta L({\bx}_n^i)} \\
				& \hspace{1cm} =\frac{1}{N}\sum_{i=1}^Ne^{-\beta L({\bx}_n^i)}(e^{-\beta(L({\bx}_{n+1}^i)-L({\bx}_n^i))}-1)\\
				&  \hspace{1cm}  \ge\frac{1}{N}\sum_{i=1}^Ne^{-\beta L({\bx}_n^i)}(-\beta)(L({\bx}_{n+1}^i)-L({\bx}_n^i))\\
				&  \hspace{1cm}  \ge-\frac{\beta}{N}\sum_{i=1}^Ne^{-\beta L({\bx}_n^i)}C_L\|{\bx}_{n+1}^i-{\bx}_n^i\|\ge-\frac{\beta C_Le^{-\beta L_\sharp}}{N}\sum_{i=1}^N\|{\bx}_{n+1}^i-{\bx}_n^i\|.
			\end{aligned}
		\end{align}
		Now, we use \eqref{aa2-1-0}, \eqref{aa7} and \eqref{T3.6P2} to see
		\begin{align}\label{T1.4peq2}
			\begin{aligned}
				&\mathbb{E}\|{\bx}_{n+1}^i-{\bx}_n^i\|\\
				&\hspace{1cm}=\mathbb{E}\|-\lambda_0h({\bx}_n^i-{\bf M}_{\beta,n})-\lambda_1h(\bar{\bx}_n-{\bf M}_{\beta,n})-\sigma({\bx}_n^i-{\bf M}_{\beta,n})\odot {\bf W}_{h,n}\|\\
				&\hspace{1cm}\le\lambda_0h\mathbb{E}\|{\bx}_n^i-{\bf M}_{\beta,n}\|+\lambda_1h\mathbb{E}\|\bar{\bx}_n-{\bf M}_{\beta,n}\|+\sigma\mathbb{E}\|({\bx}_n^i-{\bf M}_{\beta,n})\odot {\bf W}_{h,n}\|\\
				&\hspace{1cm}\le\lambda_0h\mathbb E\|{\bx}_n^i-{\bf M}_{\beta,n}\|+\frac{\lambda_1h}{N}\sum_{j=1}^N\mathbb E\|{\bx}_n^j-{\bf M}_{\beta,n}\|+\sigma\mathbb{E}\|({\bx}_n^i-{\bf M}_{\beta,n})\odot {\bf W}_{h,n}\|\\
				&\hspace{1cm}\le2(\lambda_0+\lambda_1)h\tilde C\alpha^n+2\sigma\sqrt{h\tilde C}\Big((1-\lambda_0h)^2+\sigma^2\Big)^{n/2}.
			\end{aligned}
		\end{align}
		\noindent We sum up \eqref{T1.4peq1} over $n$ and take expectation to find
		\begin{align*}
			\mathbb{E}\left[\frac{1}{N}\sum_{i=1}^Ne^{-\beta L({\bx}_n^i)}\right]&\ge\mathbb{E}\left[\frac{1}{N}\sum_{i=1}^Ne^{-\beta L({\bx}_0^i)}\right]-\frac{\beta C_Le^{-\beta L_\sharp}}{N}\sum_{k=0}^{n-1}\sum_{i=1}^N\mathbb{E}\|{\bx}_{k+1}^i-{\bx}_{k}^i\|\\
			&\hspace{-2.2cm}\ge\mathbb E\left[e^{-\beta L({\bx}_0)}\right]-\beta C_Le^{-\beta L_\sharp}\left(\frac{2(\lambda_0+\lambda_1)h\tilde C}{1-\alpha}+\frac{2\sigma\sqrt{h\tilde C}}{1-\sqrt{(1-\lambda_0h)^2+\sigma^2}}\right),
		\end{align*}
		where we have used \eqref{T1.4peq2} in the second inequality.
		As $n\to\infty$, we obtain
		\[\mathbb{E}\big[e^{-\beta L({\bx}_\infty)}\big]\ge \varepsilon\mathbb{E}\left[e^{-\beta L({\bx}_0)}\right].\]
		We take the logarithm on both sides of the above inequality to get
		\[-\frac{1}{\beta}\log\mathbb{E}\big[e^{-\beta L({\bx}_\infty)}\big]\le-\frac{1}{\beta}\log\mathbb{E}\big[e^{-\beta L({\bx}_0)}\big]-\frac{1}{\beta}\log\e.\]
		We also combine the above relation and
		\[\essinf_{\omega\in\Omega}L({\bx}_\infty(\omega))=-\frac{1}{\beta}\log e^{-\beta\essinf_{\omega\in\Omega}L({\bx}_\infty(\omega))}\le-\frac{1}{\beta}\log\mathbb{E}\big[e^{-\beta L({\bx}_\infty(\omega))}\big]\]
		to derive
		\[\essinf_{\omega\in\Omega}L({\bx}_\infty(\omega))\le-\frac{1}{\beta}\log\mathbb{E}\big[e^{-\beta L({\bx}_0)}\big]-\frac{1}{\beta}\log\e.\]
		Now, we use Laplace principle \cite{R-S} to derive
		\[\lim_{\beta\to\infty}-\frac{1}{\beta}\log\mathbb{E}\big[e^{-\beta L({\bx}_0)}\big]=\essinf_{\omega\in\Omega}L({\bx}_0(\omega)).\]
		So, if we define
		\[E(\beta):=-\essinf_{\omega\in\Omega}L({\bx}_0(\omega))-\frac{1}{\beta}\log\mathbb{E}\big[e^{-\beta L({\bx}_0)}\big]-\frac{1}{\beta}\log\e,\]
		we have
		\[\essinf_{\omega\in\Omega}L({\bx}_\infty)\le\essinf_{\omega\in\Omega}L({\bx}_0(\omega))+E(\beta),\]
		where $\lim_{\beta\to\infty}E(\beta)=0$. This yields our desired result.
	\end{proof}
	\begin{remark}
		\begin{enumerate}
			\item While the inequality condition in Theorem \ref{T3.7} (ii) requires the step size $h$ to be sufficiently small, condition \eqref{cond_para} implies that $h$ cannot be chosen arbitrarily small independently of $\sigma$. Specifically, condition \eqref{cond_para} restricts $h$ to the interval $(\frac{1 - \sqrt{1 - \sigma^{2}}}{\lambda_{0}}, \frac{1 + \sqrt{1 - \sigma^{2}}}{\lambda_{0}})$ assuming $\sigma < 1$. Thus, this assumption acts as a standard step size restriction that must be carefully balanced with the noise intensity, rather than a structural constraint on the dynamics.
			\item If a global minimizer of $L$ is contained in the support of ${\bx}_0$, then
			\[\essinf_{\omega\in\Omega}L({\bx}_\infty(\omega))\le L_\sharp +E(\beta).\]
		\end{enumerate}
	\end{remark}
	
	
	\section{Numerical Simulations}\label{sec:4}
	\setcounter{equation}{0}
	
	In this section, we provide several numerical results using non-stochastic/stochastic CBO and Ad-CBO algorithms for static and dynamic objective functions.
	
	\subsection{Static case} We evaluate the performance of Ad-CBO algorithm comparing with non-stochastic and stochastic CBO algorithms for the static case by testing its efficiency in finding a minimum point of the Rastrigin function, which is used by Adam-CBO in \cite{CJL}. The Rastrigin function $f:\mathbb R^d\to\mathbb R$ is defined by
	\begin{align}\label{Rastrigin}
		f(\bx)=\frac{1}{d}\sum_{i=1}^d\Big((x^i)^2-10\cos(2\pi x^i)+10\Big),\quad\bx=(x^1,\cdots,x^d)\in\mathbb R^d.
	\end{align}
	This function has at least $2^d-1$ local minimum points around the unique global minimum, the \textit{origin}, which is a good candidate for testing an optimization problem. In Figure \eqref{fig:rast1}, one can see a 3-dimensional picture of the Rastrigin function on $\mathbb R^2$, and Figure \eqref{fig:rast2} shows the local minima over the contour plot of it.
	
	\begin{figure}[htbp]
		\centering
		\begin{subfigure}[b]{0.5\textwidth}
			\begin{overpic}
				[width=1\linewidth]{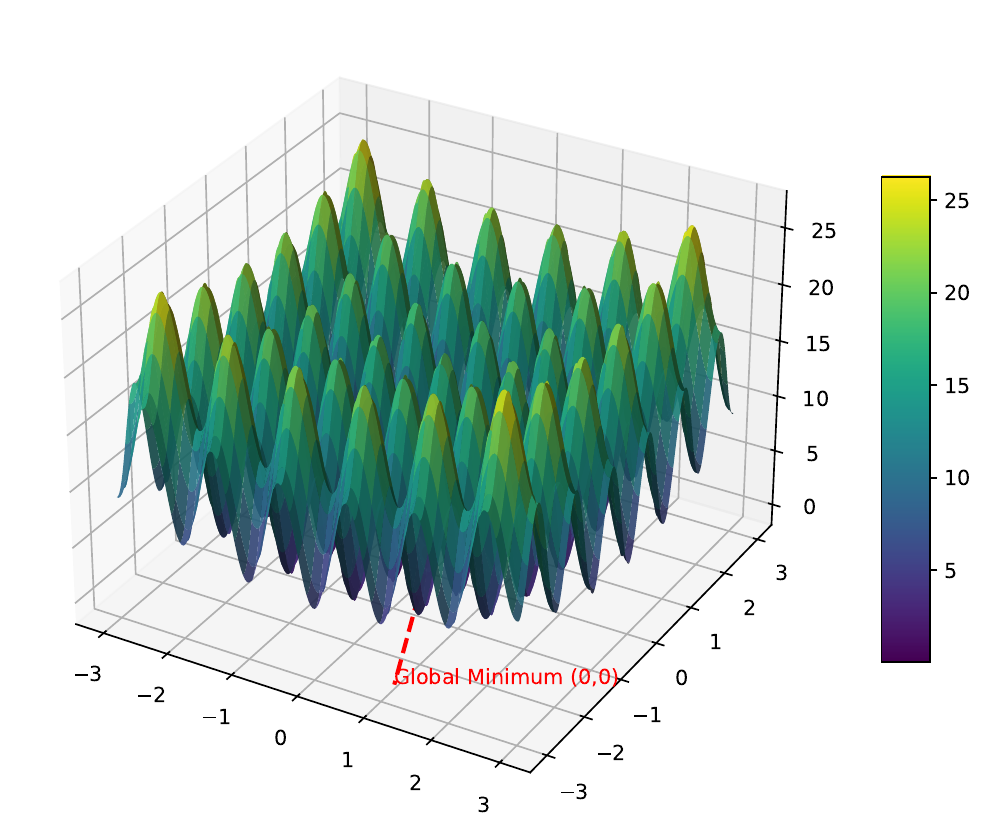}
				\put(-1.1,40){\rotatebox{100}{\small $f(x,y)$}}
				\put(25,7){\small $x$}
				\put(70,10){\small $y$}
			\end{overpic}
			\caption{Graph of Rastrigin function}
			\label{fig:rast1}
		\end{subfigure}
		\begin{subfigure}[b]{0.5\textwidth}
			\begin{overpic}
				[width=0.9\textwidth]{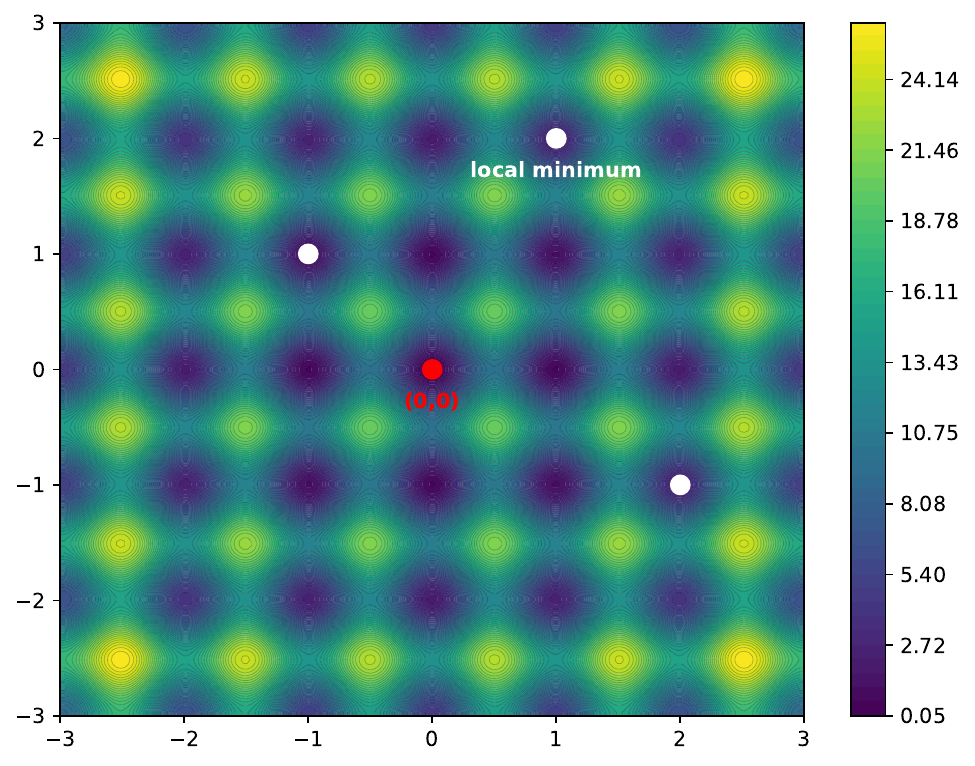}
				\put(12,-1.5){\small $x$}
				\put(-2,10){\small $y$}
				\put(101,40){\small $f(x,y)$}
			\end{overpic}
			\subcaption{{Contours of Rastrigin function}}
			\label{fig:rast2}
		\end{subfigure}
		\caption{Graphical representation of \eqref{Rastrigin}  with $d=2$}
	\end{figure}
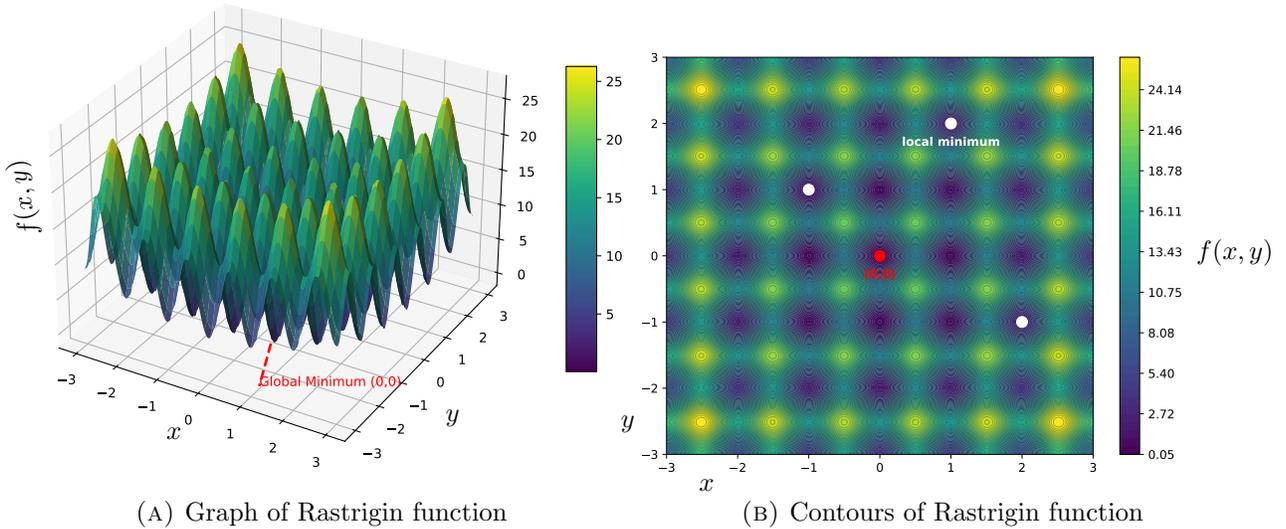
	Here, we clarify the name of algorithms. In this section, we call CBO algorithm for \eqref{discrete} with $\lambda_1=0$, i.e.,
	\[\bx_{n+1}^i=\bx_n^i-\lambda_0h(\bx_n^i-{\bf M}_{\beta,n})-\sigma(\bx_n^i-{\bf M}_{\beta,n})\odot{\bf W}_{h,n},\]
	and Ad-CBO algorithm for \eqref{discrete} with $\sigma=0$, i.e.,
	\[\bx_{n+1}^i=\bx_n^i-\lambda_0h(\bx_n^i-{\bf M}_{\beta,n})-\lambda_1h(\bar{\bx}_n-{\bf M}_{\beta,n}).\]
	We also consider Adam-CBO in numerical tests. It is based on Adam optimizer \cite{KB} with the gradient part of CBO algorithm. Since the Adam optimizer is a powerful tool for minimization problems, {\color{black}incorporating it into the CBO algorithm can potentially improve its search performance \cite{CJL}. Hence, our results include the comparison between the CBO algorithm with average and Adam-CBO algorithm.} The Adam-CBO algorithm is given as \cref{adam_unified_algo}. The decay rate parameters are chosen to be $\beta_1=0.9$ and $\beta_2=0.99$ referring to \cite{KB}.
	{\color{black}
		\begin{algorithm}
			\begin{algorithmic}
				\Require $h$: Stepsize
				\Require $\beta_1,\beta_2\in[0,1)$: Exponential decay rates for the moment estimates
				\Require $L_n(\bx)$ or $L(\bx)$: Objective function (time-varying or static)
				\Require $\bx_0=[\bx_0^1,\cdots,\bx_0^N]\in\mathbb R^{d\times N}$: Initial data
				\Require $\texttt{dynamic\_mode}\in\{\texttt{True}, \texttt{False}\}$: Whether the algorithm is dynamic
				\Require $\epsilon_{\mathrm{tol}}$: Convergence tolerance threshold (only used in static mode)
				\Require $T$: Maximum number of iterations (only used in dynamic mode)
				\Require $\pi$: Projection operator onto constraint set $\mathcal S$ (only used in dynamic mode)
				
				\State $\bm_0=[\bm_0^1,\cdots,\bm_0^N]\in\mathbb R^{d\times N}\gets{\bf 0}$ \Comment{Initialize $1^{st}$ moment vectors}
				\State $\bv_0=[\bv_0^1,\cdots,\bv_0^N]\in\mathbb R^{d\times N}\gets{\bf 0}$ \Comment{Initialize $2^{nd}$ moment vectors}
				\State $n\gets0$ \Comment{Initialize timestep}
				\While{(\texttt{dynamic\_mode} \textbf{and} $n<T$) \textbf{or} (\textbf{not} \texttt{dynamic\_mode} \textbf{and} $\max_{1\le i,j\le N}\|\bx_n^i-\bx_n^j\|>\epsilon_{\mathrm{tol}}$)}
				\State $n\gets n+1$
				\For{$i=1,\cdots,N$}
				\State $\bg_n^i\gets\lambda_0(\bx_{n-1}^i-{\bf M}_{\beta,n-1})$ \Comment{CBO algorithm}
				\State $\bm_n^i\gets\beta_1\cdot\bm_{n-1}^i+(1-\beta_1)\cdot\bg_n^i$
				\State $\bv_n^i\gets\beta_2\cdot\bv_{n-1}^i+(1-\beta_2)\cdot\bg_n^i\odot\bg_n^i$
				\State $\hat{\bm}_n^i\gets\bm_n^i/(1-\beta_1)$
				\State $\hat{\bv}_n^i\gets\bv_n^i/(1-\beta_2)$
				\State $\by_n^i\gets\bx_{n-1}^i-h\cdot\hat{\bm}_n^i/(\sqrt{\hat{\bv}_n^i}+10^{-6})$ \Comment{Element-wise division}
				\If{\texttt{dynamic\_mode}}
				\State $\bx_n^i\gets\pi(\by_n^i)$ \Comment{Projection in dynamic case}
				\Else
				\State $\bx_n^i\gets\by_n^i$
				\EndIf
				\EndFor
				\EndWhile
			\end{algorithmic}
			\caption{Unified Adam-CBO algorithm (Static and Dynamic cases)}\label{adam_unified_algo}
		\end{algorithm}
	}
	
	For numerical examples, we set $d=15$ and use $N=50$ particles, which are uniformly randomly distributed on $[2,4]^d$ at the initial time ($n=0$). Note that this ensures initial configurations do not contain the optimal point, which makes the optimization problem to be more challenging. We set $\beta=100$, a step size of $h=0.1$ and $\epsilon_{\mathrm{tol}}={\color{black}10^{-6}}$ for tolerance. Given the tolerance value, we regard particles make consensus if
	\[\max_{l\in[d]}\left(\max_{i,j\in[N]}|x_n^{i,l}-x_n^{j,l}|\right)<\epsilon_{\mathrm{tol}}\quad\mbox{where}\quad\bx_n^i=(x_n^{i,1},\cdots,x_n^{i,d}),\]
	and stop the iteration.
	
	To evaluate the consistency of the algorithm's results upon random initial data and noise, we make 50 simulations for each algorithm and calculate the mean and variance from iterations. We denote the consensus point as $\bx_{\infty,k}$ and the iteration numbers as $n_k$ until the consensus in the $k$-th simulation. For each algorithm, we obtain
	\[\left\{\bx_{\infty,k}\right\}_{k=1}^{50}\quad\mbox{and}\quad\left\{n_k\right\}_{k=1}^{50}.\]
	{\color{black}Here, we briefly explain our choice of the parameters $\sigma$ and $\lambda_1$ in the experiment. Considering the terms
		\[\lambda_1h(\bar\bx_n-{\bf M}_{\beta,n}),\qquad\sigma(\bx_n^i-{\bf M}_{\beta,n})\odot{\bf W}_{h,n},\]
		one sees that equal values of $\lambda_1$ and $\sigma$ yield similar magnitudes, as ${\bf W}_{h,n}$ has elements $\sim N(0,h)$. In any case, particles must reach consensus to achieve optimization. While Ad-CBO (without noise) always succeeds, CBO with noise is non-trivial.
		
		For particles $\{\bx_n^i\}_{i\in[N],n\ge0}$ generated by discrete CBO,
		\begin{align*}
			x_n^{i,l}-x_n^{j,l}=(x_0^{i,l}-x_0^{j,l})\Pi_{k=0}^{n-1}(1-h\lambda_0-\sigma W_{h,k}^l).
		\end{align*}
		By the law of large numbers, 
		\begin{align}\label{lln}
			\lim_{n\to\infty}\frac{1}{n}\log|x_n^{i,l}-x_n^{j,l}|=\mathbb E\Big[\log|1-h\lambda_0-\sigma\sqrt{h}Z|\Big],\quad Z\sim N(0,1),
		\end{align}
		so the sign of the expectation determines convergence: negative implies almost sure consensus, positive implies failure. We define
		\begin{align*}
			\Lambda(\sigma):=\mathbb E\Big[\log|1-h\lambda_0-\sigma \sqrt{h}Z|\Big]=\log(\sigma\sqrt{h})+\mathbb E\left[\log\Big|Z-\frac{1-h\lambda_0}{\sigma\sqrt{h}}\Big|\right].
		\end{align*}
		\begin{figure}[htbp]
			\centering
			\begin{overpic}
				[width=0.6\linewidth]{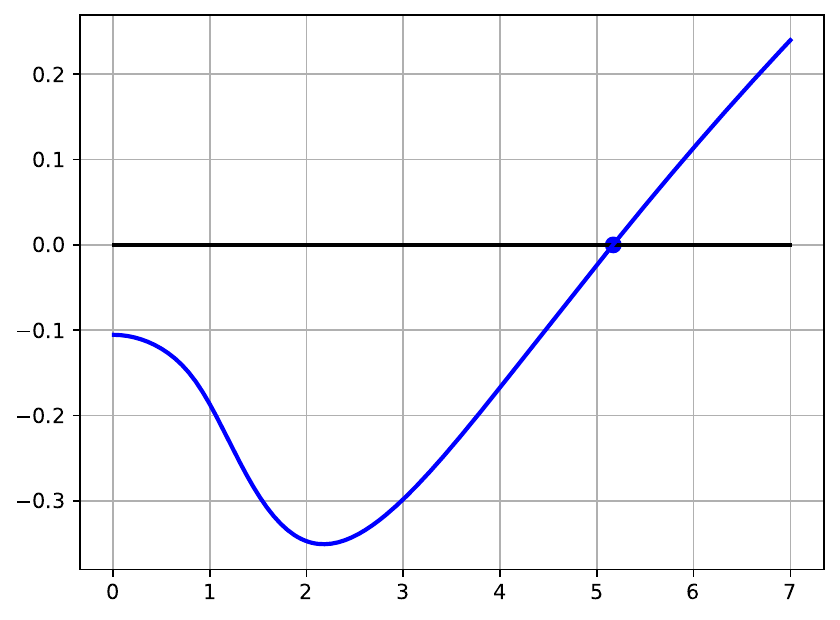}
				\put(53,-0.5){$\sigma$}
				\put(-5,39){$\Lambda(\sigma)$}
				\put(74,40){\bf\color{blue}$\sigma^*$}
			\end{overpic}
			\caption{Determinant value $\Lambda(\sigma)$ when $h=0.1$ and $\lambda_0=1$.}
			\label{Lambda_graph}
		\end{figure}
		In our experiments, we set $h=0.1$ and $\lambda_0=1$. \Cref{Lambda_graph} shows $\Lambda(\sigma)$, indicating that CBO with noise achieves consensus for $0\le\sigma<\sigma^*$ where $\sigma^*$ is the root of $\Lambda(\sigma)$ with $\sigma^*\approx5.17$. Accordingly, we choose $\sigma = 1, \dots, 5$ and $\lambda_1 = 1, \dots, 5$.
	}
	\begin{center}
			\begin{tabular}{|l|c|c|c|c|c|c|}
				\noalign{\hrule height 0.5pt}
				CBO ($\lambda_1=0$)\T\B& \cellcolor{CBO}$\sigma=0$ &\cellcolor{CBO} $\sigma=1$\T&\cellcolor{CBO} $\sigma=2$&\cellcolor{CBO} $\sigma=3$ &\cellcolor{CBO} $\sigma=4$ &\cellcolor{CBO} $\sigma=5$\\
				\hline\hline
				$\mathbb E\left[L(\bx_\infty)\right]$\T& 12.315 & 11.752 & 10.963 & 9.781 & 9.591 & 11.297\\
				$\mbox{Var}\left[L(\bx_\infty)\right]$ & 1.447 & 1.075 & 1.443 & 1.679 & 1.756 & 2.107\\
				$\mathbb E\left[\sharp\mbox{ of iterations}\right]$\T& 94 & 94.92 & 66.66 & 94.42 & 225.44 & 3738.38\\
				$\mbox{Var}\left[\sharp\mbox{ of iterations}\right]$ & 0 & 307.914 & 261.944 & 1070.084 & 8000.126 & $1.553\times 10^6$\\
				\hline\hline
				{Ad-CBO ($\sigma=0$)}\T\B&\cellcolor{adCBO}$\lambda_1=1$\T&\cellcolor{adCBO}$\lambda_1=2$ &\cellcolor{adCBO} $\lambda_1=3$ &\cellcolor{adCBO} $\lambda_1=4$ &\cellcolor{adCBO} $\lambda_1=5$ & \cellcolor{adamCBO}Adam-CBO\\
				\hline\hline
				$\mathbb E\left[L(\bx_\infty)\right]$\T& 9.202 & 8.138 & 7.717 & 7.477 & 7.176 & 8.961\\
				$\mbox{Var}\left[L(\bx_\infty)\right]$ & 0.988 & 0.709 & 0.957 & 0.915 & 0.837 & 0.441\\
				$\mathbb E\left[\sharp\mbox{ of iterations}\right]$\T& 94 & 94 & 94 & 94 & 94 & 214.1\\
				$\mbox{Var}\left[\sharp\mbox{ of iterations}\right]$ & 0 & 0 & 0 & 0 & 0 & 83.706\\
				\hline
			\end{tabular}
		\captionof{table}{Simulation Results from Solving Rastrigin Function by CBO and Ad-CBO: Statistics based on fifty simulations with the same parameter setting ($N=50,~h=0.1,~\lambda_0=1,~\beta=100,~\epsilon_{\mathrm{tol}}=10^{-6}$)}
		\label{new_sim_rast}
	\end{center}
	\begin{figure}
		\centering
		\begin{overpic}
			[width=0.7\linewidth]{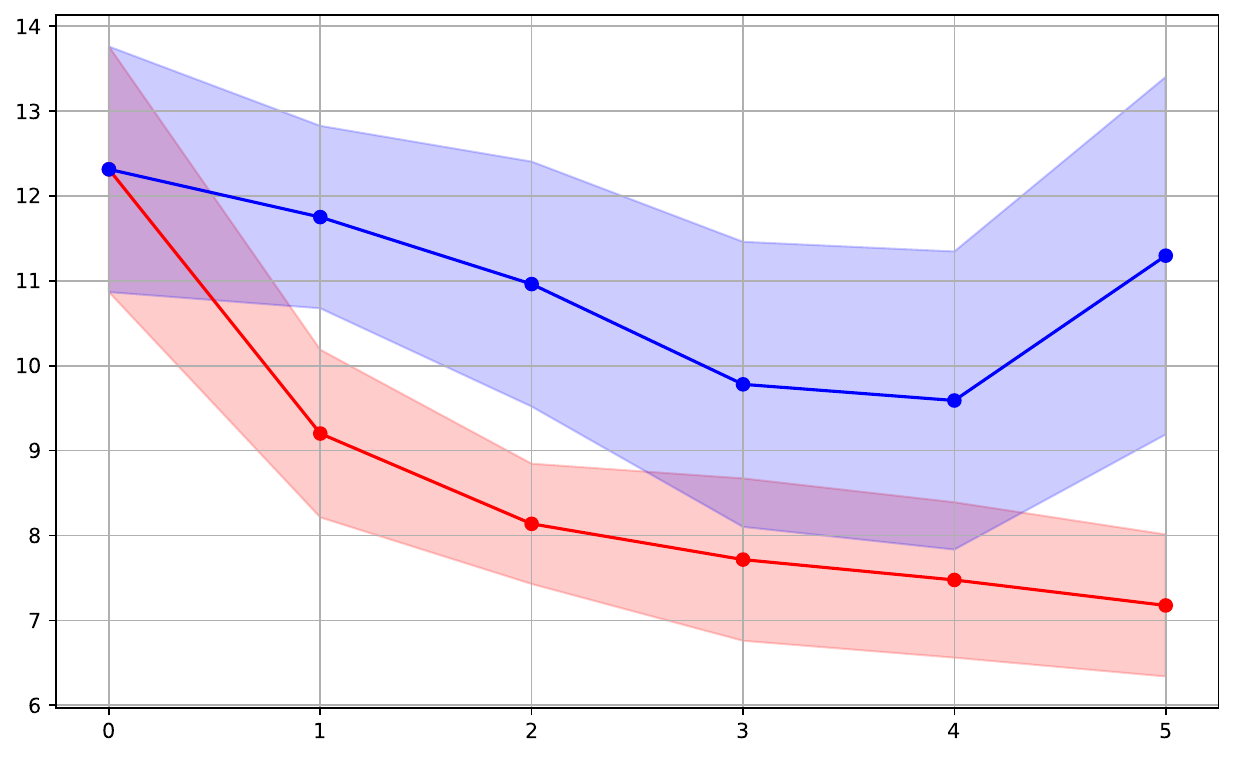}
			\put(40,-1.8){$\lambda_1$(red), $\sigma$(blue)}
			\put(-3.2,13){\rotatebox{90}{$\mathbb E[L(\bx_\infty)]\pm\mbox{SD}[L(\bx_\infty)]$}}
		\end{overpic}
		\caption{Confidence interval for $L(\bx_\infty)$ obtained by CBO with noise (blue) and Ad-CBO (red) according to $\lambda_1$ and $\sigma$, respectively. The solid lines with dots represent the expectations.}
		\label{conf_int}
	\end{figure}
	
	In Table \ref{new_sim_rast}, we summarize results by providing the mean and variance of {\color{black}$\left\{L(\bx_{\infty,k})\right\}_{k=1}^{50}$} and $\left\{n_k\right\}_{k=1}^{50}$, respectively. After many tests of choosing adequate parameter values, we use $\beta=100$ for controlling an efficient drift toward the minimum point. In Table \ref{new_sim_rast}, we tabulate results from various tests with different $\sigma$ and $\lambda_1$. {\color{black}For readers' convenience, we plot the confidence interval for $L(\bx_\infty)$ in \Cref{conf_int}.} The result from CBO ($\sigma = 0$) is given as a benchmark in the top-left panel of Table \ref{new_sim_rast}.
	
	{\color{black}
		For CBO with noise, one observes that increasing $\sigma$ initially leads to better optimization performance. However, once $\sigma$ exceeds a certain threshold, the performance degrades. This is due to the fact that excessive noise hinders convergence, as the risk of losing good candidates outweighs the potential gain from exploring new areas. Hence, although convergence is theoretically guaranteed ($\sigma<\sigma^*$), $\sigma$ cannot be chosen arbitrarily large.
		
		In terms of convergence speed, the case $\sigma=2$ achieves the fastest convergence among cases, in line with our earlier analysis of $\Lambda(\sigma)$. The corresponding average number of iterations is $66.66$, which appears smaller than the constant $94$ iterations of Ad-CBO. Nevertheless, the variance of $261.944$ indicates that this behavior is highly unstable: the actual number of iterations fluctuates significantly, and in some cases convergence may take considerably longer.
		
		By contrast, Ad-CBO exhibits steadily improving optimization performance as $\lambda_1$ increases. Its variance remains low, and the average optimization value is consistently better than that of CBO with noise. Moreover, the number of iterations is always fixed at $94$, matching that of CBO without noise. This provides not only superior optimization results but also predictable convergence behavior.
		
		Finally, when compared to Adam-CBO, an Adam-type variant known for strong performance in optimization, Ad-CBO yields better average optimization values for $\lambda_1 \geq 2$. Although Adam-CBO achieves smaller variance, the confidence intervals indicate that Ad-CBO remains the superior method overall. The advantage is further supported by iteration counts, where Ad-CBO clearly outperforms. Regarding computational cost, we note that the inclusion of $\bar{\bx}_n$ does not increase computational complexity beyond $\mathcal{O}(n)$, meaning that Ad-CBO matches or even improves upon the efficiency of CBO.}
	
	\subsection{Dynamic case}\label{sec_dynamics}
	In another numerical test, we investigate the performance of Ad-CBO in real problems of tracking a moving optimal point. As a real online problem, we consider an \textit{online portfolio selection (OPS) problem} of an investor, who searches for the best combination or portfolio of financial assets. We investigate the performance of Ad-CBO relative to CBO and Adam-CBO according to the best and average values of our objective function by each algorithm, and variances. Furthermore, the best portfolio's volatility, i.e., the standard deviation of $\{{\br}_n\cdot\bx_n^{\mathrm{best}}\}_n$,
	where $\bx_n^{\mathrm{best}}$ is chosen by
	\[\bx_n^{\mathrm{best}}\in\left\{\bx_n^i~:~L(\bx_n^i)=\min_{j\in[N]}L(\bx_n^j)\right\},\]
	is considered for a safe investment strategy. {\color{black}Here, ${\br}_n$ denotes a price relative vector which will be defined in the next paragraph.} Based on the results, we discuss the efficiency of Ad-CBO algorithm to be applied for resolving problems of a real stock market. 
	
	
	The OPS problem is a sequential capital allocation in an asset pool, pursuing the maximum final wealth in the long run \cite{L-H}. For certain time period from $t_0$ to $t_M$, an online portfolio investor renews her wealth allocation over $d$ assets at each time $t_n$ for $n=0,\cdots,M$. We denote some quantity at $t_n$ as a subscript of $n$. The market price change during time period $[t_{n-1},t_n]$ is denoted by a $d$-dimensional price relative vector ${\br}_n$:
	\begin{align*}
		{\br}_n = (r_n^1, \cdots, r_n^d)\in \mathbb{R}^d_{+},\quad n \in\mathbb N,
	\end{align*}
	where $r_n^i$ is the (price) return of the $i$-th asset. In other words, the return vector ${\br}_n$ is defined by the ratio of closing prices between $n-1^{th}$ and $n^{th}$ day, i.e.,
	\[r_n^i=\frac{\mbox{Price of $i$-th asset at time }t_n}{\mbox{Price of $i$-th asset at time }t_{n-1}},\quad\forall~i\in[d].\]
	Here, we assume ${\br}_n$ and ${\br}_{n'}$ are independent if $n\ne n'$. An investor makes a portfolio decision at time $t_n$ based on historical return data ${\br}_{n-1}$, before we see closing prices and realized returns.
	
	Referring to the optimal portfolio selection problem of \cite{BHKLMY}, our test is designed to find the best portfolio weight of six stocks:
	\[\mbox{AAPL},\quad\mbox{MSFT},\quad\mbox{SBUX},\quad\mbox{DB},\quad\mbox{GLD},\quad\mbox{TSLA.}\]
	The first two stocks, Apple, Inc. (AAPL) and Microsoft (MSFT) are indexed on the Dow Jones Industrial Average Index\footnote{The Dow Jones Industrial Average, Dow Jones, or simply the Dow, is a stock market index of 30 prominent companies listed on stock exchanges in the United States.}. We also consider stocks of two most-followed stock market indices along with the DJIA in the United States by including Starbucks (SBUX) and Tesla (TSLA) of the Nasdaq Composite, and Deutsche Bank Aktiengesellschaft (DB) of the New York Stock Exchange (NYSE). Finally, we include SPDR Gold Shares (GLD) to reflect the performance of the price of gold bullion.
	
	By using historical returns from the first day of 2017 to the last day of 2019, we set 692 windows which are equipped with mean and covariance matrix of return over consecutive 60 days. Windows are rolled by one day over the sample period, denoted by
	\begin{align*}
		\mu_n=\mathbb E\left[\left\{{\br}_k\right\}_{k=n-59}^{n}\right]\quad\mbox{and}\quad\Sigma_n=\mbox{Cov}\left[\left\{{\br}_k\right\}_{k=n-59}^{n}\right],\quad\forall~n=0,\cdots,691.
	\end{align*}
	Following \cite{M}, an investor's objective is to find the best combination of weights on assets, which yields the highest Sharpe ratio. Hence, in the setting of our numerical simulation, the objective function is the negative Sharpe ratio to be minimized as suggested in \eqref{B-1}, i.e.,
	\[L_n(\bx):=-\frac{\mu_n\cdot\bx}{\sqrt{\bx\Sigma_n\bx^\intercal}},\quad\forall~n=0,\cdots,691.\]
	Since the portfolio selection procedure occurs before a current window's return rates are realized, the objective function $L_n$ is used to update $\bx_n$ to get $\bx_{n+1}$.
	
	By considering a particle $\bx_n^i$ as one of alternative weights on an asset, we assume that an element is between $0$ and $1$ and they are summed {\color{black}up to} 1. We impose $\bx_n^i$ to be in the admissible set $\mathcal S$ defined by
	\[\mathcal S:=\left\{(x_1,\cdots,x_d)~:~x_i\ge0,~\forall~i\in[d]\quad\mbox{and}\quad\sum_{i=1}^dx_i=1.\right\}\]
	In short, (Ad-)CBO algorithm is applied to solve the OPS problem as follows:
	\begin{align*}
		\begin{dcases}
			\by_{n+1}^i=\bx_n^i-\lambda_0h(\bx_n^i-{\bf M}_{\beta,n})-\lambda_1h(\bar{\bx}_n-{\bf M}_{\beta,n})-\sigma({\bx}_n^i-{\bf M}_{\beta,n})\odot{\bf W}_{h,n},\\
			\bx_{n+1}^i=\pi(\by_{n+1}^i),\quad\forall~n=0,\cdots,691,
		\end{dcases}
	\end{align*}
	where $\pi:\mathbb R^d\to\mathcal S$ is a projection map \cite{W-CP} and
	\begin{align}\label{CBO_port}
		\begin{aligned}
			{\bf M}_{\beta,n}=\frac{\sum_{i=1}^Ne^{-\beta L_n(\bx_n^i)}\bx_n^i}{\sum_{i=1}^Ne^{-\beta L_n(\bx_n^i)}}.
		\end{aligned}
	\end{align}
	Here, we use the projection map introduced in \cite{W-CP}. So, what we call CBO is \eqref{CBO_port} with $\lambda_1=0$ and Ad-CBO is \eqref{CBO_port} with $\sigma=0$. We use $\beta_1=0.9$ and $\beta_2=0.99$ as in the static case.
	
	We denote the $i$-th portfolio weight at $n$-th date in $k$-th simulation for $k = 1, \dots, 50$ by $\bx_{k,n}^i$. We summarize the results according to expectation and variance of the best and average approximate, respectively. At every step in each simulation, we choose the best {\color{black}particle $\bx_{k,n}^{\mathrm{best}}$} such as
	\[L(\bx_{k,n}^{\mathrm{best}})=\min_{i\in[N]}L(\bx_{k,n}^i),\]
	and get the average value for all dates such that
	{\color{black}\[L_k^{\mathrm{best}}:=\frac{1}{692}\sum_{n=0}^{691}L(\bx_{k,n}^{\mathrm{best}}),\quad\forall~k=1,\cdots,50.\]}
	The expectation and variance of {\color{black}$\left\{L_k^{\mathrm{best}}\right\}_{k=1}^{50}$} summarize the results of the best values. Next, we examine such estimates with the average value, i.e., the expectation and variance of $\left\{\overline{L(\bx_k)}\right\}_{k=1}^{50}$ where
	\[\overline{L(\bx_k)}=\frac{1}{692}\sum_{n=0}^{691}\left(\frac{1}{N}\sum_{i=1}^NL(\bx_{k,n}^i)\right),\quad\forall~k=1,\cdots,50.\] Finally, we estimate the portfolio's volatility to see what is the stability of the portfolio suggested by each algorithm given unexpected shocks to financial markets. The portfolio volatility is estimated by the standard deviation of $\left\{{\br}\cdot{\bx}_k\right\}_{k=1}^{50}$, which is defined by
	\[{\br}\cdot{\bx}_k:=\frac{1}{692}\sum_{n=0}^{691}{\br}_n\cdot\left(\frac{1}{N}\sum_{i=1}^N\bx_{k,n}^i\right),\quad\forall~k=1,\cdots,50.\]
	
	\begin{center}
		\begin{tabular}{|l|c|c|c|c|}
				\hline
				CBO ($\lambda_1=0$)\T\B
				& \multicolumn{1}{c|}{\cellcolor{CBO}$\sigma=0$}
				& \multicolumn{1}{c|}{\cellcolor{CBO}$\sigma=1$}
				& \multicolumn{1}{c|}{\cellcolor{CBO}$\sigma=2$}
				& \multicolumn{1}{c|}{\cellcolor{CBO}$\sigma=3$}\\
				\hline\hline
				$\mathbb E\left[L^{\mathrm{best}}\right]$ \T & -1.787 & -1.754 & -1.743 & -1.773\\
				$\mbox{Var}\left[L^{\mathrm{best}}\right]$ & 0.0127 & 0.0240 & 0.0337 & 0.0635\\
				$\mathbb E\left[\overline{L(\mbox{\boldmath $x$})}\right]$ & -1.739 & -1.698 & -1.693 & -1.671\\
				$\mbox{Var}\left[\overline{L(\mbox{\boldmath $x$})}\right]$ & 0.0121 & 0.0258 & 0.0350 & 0.0667 \\
				Portfolio's Volatility & 0.0176 & 0.0216 & 0.0289 & 0.0456 \\
				\hline\hline
				Ad-CBO ($\sigma=0$)\T\B 
				& \multicolumn{1}{c|}{\cellcolor{adCBO}$\lambda_1=1$}
				& \multicolumn{1}{c|}{\cellcolor{adCBO}$\lambda_1=2$}
				& \multicolumn{1}{c|}{\cellcolor{adCBO}$\lambda_1=3$}
				& \multicolumn{1}{c|}{\cellcolor{adamCBO}Adam-CBO}\\
				\hline\hline
				$\mathbb E\left[L^{\mathrm{best}}\right]$\T & -1.883 & -1.899 & -1.897 & -3.028\\
				$\mbox{Var}\left[L^{\mathrm{best}}\right]$ & 0.0049 & 0.0027 & 0.0013 & 0.0346\\
				$\mathbb E\left[\overline{L(\mbox{\boldmath $x$})}\right]$ & -1.849 & -1.872 & -1.875 & -0.745\\
				$\mbox{Var}\left[\overline{L(\mbox{\boldmath $x$})}\right]$ & 0.0048 & 0.0026 & 0.0013 & 0.0844\\
				Portfolio's Volatility & 0.0123 & 0.0085 & 0.0112 & 0.0589\\
				\hline
		\end{tabular}
		\captionof{table}{Results from Solving the Online Portfolio Problem by CBO, CBO with noise and Ad-CBO: Statistics based on fifty simulations with the same parameter setting ($N=20,~h=0.1,~\lambda_0=1,~\beta=1-0$)}
		\label{new_sim_port}
	\end{center}
	
	In Table \ref{new_sim_port}, the results of various type of CBO's portfolio selection are summarized. According to the trends of {\color{black}$\mathbb E\left[L^{\mathrm{best}}\right]$}, Adam-CBO is the best, then Ad-CBO's results are better than them of CBO. From the lower variance of {\color{black}$\mathbb E\left[L^{\mathrm{best}}\right]$} for Adam-CBO and Ad-CBO compared to CBO, we find that Adam-CBO and Ad-CBO perform well.
	
	In the terms of portfolio's volatility, Adam-CBO exhibits the lowest quality with 0.0589. Meanwhile, CBO and Ad-CBO have different tendencies in portfolio's volatility with each others. When CBO is applied, portfolio's volatility increases with higher $\sigma$. However, in Ad-CBO, portfolio's volatility stays around 0.01 as $\lambda_1$ increase and the Ad-CBO's performance improves. Therefore, Ad-CBO's strategy can achieve higher {\color{black}Sharpe ratios} with lowering or maintaining the total risk of {\color{black}variablity of such high Sharpe ratio. Furthermore,} it is an absolutely better strategy {\color{black}from an investor's view to minimize portfolio management costs and risks.}
	\subsection{Regret bound for Ad-CBO in optimal portfolio problem}
	We discuss the \textit{regret bound} \cite{IHSYFKK} of Ad-CBO alrogithm in OPS problem introduced in the previous subsection. As the wealth changes at every time $t_n$ by ${\br}_n\cdot {\bx}_n$ in the period from $t_0$ to $t_{M}$, a portfolio strategy $\{\bx_n\}_{n=1}^M$ makes the last wealth $S_M$ to be the initial wealth $S_0$ multiplied by $ \prod_n {\br}_n\cdot {\bx}_n$, i.e.,
	\[S_{M} = S_0 \prod_{n=1}^{M} {\br}_n\cdot{{\bx}_n} = S_0 \prod_{n=1}^{M} \sum_{i=1}^{d}  r_{n}^i  x_n^i.\]
	We use the wealth $S_M$ to estimate the performance of the suggested algorithm. In literature, there are various approaches to maximize the final wealth: ``Buy-And-Hold", ``Follow-the-Winner" to ``Meta-Learning Algorithms"\cite{D-G-U, L-H}. The above relation can define the exponential growth rate of the strategy $\{{\bx}_{n}\}_{n=1}^M$ for the investment during a period $[t_0,t_M]$ as
	\[\sum_{n=1}^{M}\log\left({\br}_n\cdot {\bx}_n\right) = \log S_{M} - \log S_0.\]
	
	Regret bound is the difference between exponential growth rates of wealth made by the suggested investment strategy and the best constant portfolio strategy in hindsight, which is a key estimate to evaluate an algorithm's performance in solving the OPS problem \cite{H-K}, \cite{L-H}. To the best of our knowledge, previous literature on the OPS has only provided numerical simulation results for estimating the performance of algorithms. Here, we propose the formula of regret bound $R_M$ until $M$ as follows:
	\begin{align}\label{regret}
		R_M(\bx_n):=\max_{\bx\in\mathcal S}\left(\frac{1}{M}\sum_{n=1}^M\log(1+{\br}_n\cdot\bx)-\frac{1}{M}\sum_{n=1}^M\log(1+{\br}_n\cdot\bar\bx_n)\right),\quad\forall~M\ge1.
	\end{align}
	Assume that returns are given as
	\begin{align*}
		{\br}_n=(r_n^1,\cdots,r_n^d),\quad\forall~n\ge0,
	\end{align*}
	where random variables $r_n^i$, $i\in[d]$ are i.i.d., satisfying
	\begin{align}\label{r_bound}
		r_n^i\in(\underline r,\overline r),\quad\forall~i\in[d],~\forall~n\ge0,
	\end{align}
	for some constants $-1<\underline{r}<0<\overline{r}$. If we have similar estimates with our main analysis results for the online objective function, i.e., time-dependent counterpart of Theorem \ref{T3.7}, then meaningful upper bound of the regret bound is found with $\beta=\infty$. We leave it in Appendix \ref{apdx} for readers who may be interested.
	Although more future works are needed to investigate conditions in a rigorous way, we believe that our mathematical analysis in Appendix \ref{apdx} is the first and important step for the analysis on the regret bound.
	
	\begin{figure}
		\begin{subfigure}[ht]{0.45\textwidth}
			\includegraphics[width=1.15\textwidth]{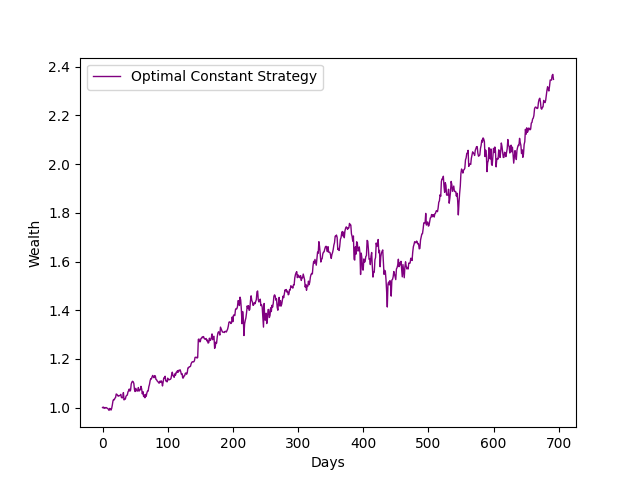} \subcaption{Benchmark by SLSQP}
		\end{subfigure}
		\hspace{.2cm}
		\begin{subfigure}[ht]{0.45\textwidth}
			\includegraphics[width=1.15\textwidth]{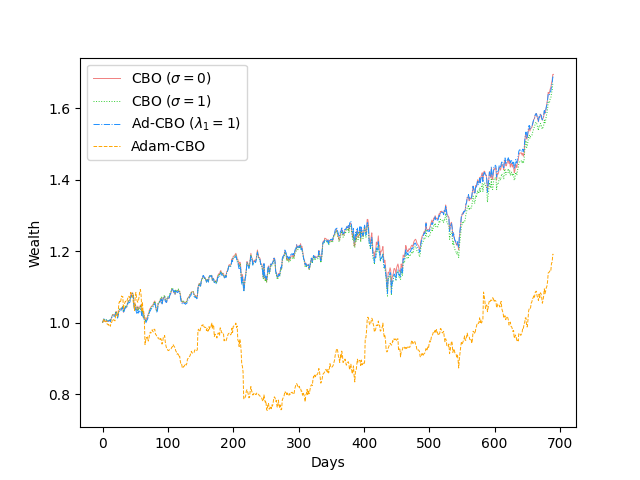}
			\subcaption{CBO, Ad-CBO and Adam-CBO}
		\end{subfigure}
		\caption{Wealth evolution: average of fifty simulations wealth based on CBO ($\sigma=0,1$), Ad-CBO ($\lambda_1=1$) and Adam-CBO ($\beta_1=0.9,\beta_2=0.99$).}
		\label{wealth_graph}
	\end{figure}
	
	To compare the efficiency of CBO, Ad-CBO and Adam-CBO in terms of regret bound, we use Sequential Least SQuares Programming(SLSQP) method to find the optimal constant strategy, i.e., argmax $\bx$ in \eqref{regret}, with the same sample data and end time $M=692$ with CBO-type algorithms. In Figure \ref{wealth_graph}, we show the wealth evolution based on the optimal constant strategy obtained SLSQP as a benchmark on the top and the wealth evolution based on CBO with or without nosie and Ad-CBO on the bottom panel, respectively. Here, the results by CBO-type algorithms are from the same simulations in Section \ref{sec_dynamics}. To be more precise, we show
	\begin{align*}
		&{\br}_n\cdot\bx\quad\mbox{in Figure \ref{wealth_graph}(A)}\quad\mbox{and}\quad\frac{1}{50}\sum_{k=1}^{50}{\br}_n\cdot\overline{\bx}_{k,n}^{\mathrm{best}}\quad\mbox{in Figure \ref{wealth_graph}(B).}
	\end{align*}
	Indeed, the wealth evolution based on this optimal constant strategy obtained by SLSQP demonstrates noticeably higher wealth after some initial period. It is no wonder as SLSQP calculates from a bird's eye view to use all information of the overall sample period to find the best portfolio whereas CBO-type algorithms only use the information until the day before at each days.
	
	In Figure \ref{wealth_graph}(B), wealth evolution by CBO and Ad-CBO have overall similar trend whereas that by Adam-CBO markedly different. Such results are also explained by the property of Adam optimizer. The principle of Adam optimizer is that particles keep their memories and find the best one through all information. So, they heavily rely on prior experience, which is not suitable to provide a good online strategy in stock market although they achieved good results in searching a best point for objective function. In Figure \ref{wealth_graph}(B), the wealth generated by Adam-CBO has a similar fluctuation with CBO and Ad-CBO in earlier period, but soon, they lost and the reasonably good wealth level is not easily attained by those particles. Such a trend reveals that Adam-CBO cannot be an appropriate algorithm for an investor to seek for a long-term strategy of wealth maximization.
	Lastly, we calculate the correlation of each CBO-type algorithm with SLSQP result. To summarize, the numeric results show that the trend of wealth obtained by CBO without noise (0.8176) and with noise (0.7957) and Ad-CBO (0.7986) are strongly correlated with the benchmark wealth whereas Adam-CBO has relatively a low correlation (0.5181). Furthermore, if an investor chooses one among the four CBO methods, Ad-CBO demonstrates superior performance on average: across 50 independent simulations under the same time-varying objective function, it achieves higher average terminal wealth while maintaining significantly lower average volatility than the other schemes.
	
	\section{Conclusion} \label{sec:5}
	In this paper, we proposed a consensus-based optimization algorithm with average drift and applied it to the online portfolio selection problem. As {\color{black}shown} in Theorem 4.1 of \cite{H-J-Kc}, when the convex hull of the initial states does not contain a global minimum of {\color{black}the} objective function, the consensus state generated by the CBO method without stochastic diffusion cannot reach the theoretical minimum. {\color{black}To address this limitation, we introduce a new ingredient, the {\it average drift}, which helps steer the swarm toward the global minimum.} The resulting Ad-CBO algorithm provides a better approximation of the optimal point for a time-varying objective function.
	
	In this paper, we propose a consensus based optimization algorithm with average drift and applied it to online portfolio selection problem. As found by Theorem 4.1 in \cite{H-J-Kc}, when the convex hull of initial states does not contain a global minimum of objective function, the consensus state by the CBO method without stochastic diffusion cannot reach the theoretical minimum of objective function. To address this limitation of the non-stochastic CBO algorithm, we introduce a new ingredient, the {\it average drift}, which helps steer the swarm toward the global minimum point. This Ad-CBO algorithm exhibits a better approximate for the optimal point of a time--varying objective function. 
	Particularly, we estimate the performance of Ad-CBO in solving the optimal portfolio selection (OPS) problem. For the OPS problem, we provide an upper bound estimate for the regret to Ad-CBO method which is far from optimal upper bound. There are several unresolved issues for the choices of the number of particles, $\beta$, and initial data. We leave these interesting issues for a future work.

	\appendix
	\section{Upper bound of regret bound ($\beta=\infty$)}\label{apdx}
	
	In this appendix, we explain how a meaningful upper bound of regret bound can be obtained and what it means, if we make estimates for online counterpart of Theorem \ref{T3.7}, i.e., consensus and optimized value, and have an assumption on $\left\{{\br}_n\right\}_{n\ge0}$ stated in Section \ref{sec:5}.\newline
	
	First of all, we propose the following notations:
	\begin{align*}
		\bx_n^\sharp\in\argmax_{\bx\in\mathcal S}({\br}_n\cdot\bx),\quad\forall~n\ge0\quad\mbox{and}\quad\bx^\sharp\in\argmax_{\bx\in\mathcal S}\sum_{n=1}^M\log(1+{\br}_n\cdot\bx).
	\end{align*}
	Recall that the regret bound $R_M$ is defined by
	\begin{align*}
		R_M(\bx_n):=\max_{\bx\in\mathcal S}\left(\frac{1}{M}\sum_{n=1}^M\log(1+{\br}_n\cdot\bx)-\frac{1}{M}\sum_{n=1}^M\log(1+{\br}_n\cdot\bar\bx_n)\right),\quad\forall~M\ge1.
	\end{align*}
	Now, we use the convexity of $f_n(\bx):=-\log({\br}_n\cdot\bx)$ to get
	\begin{align} \label{C-6}
		R_M =\frac{1}{M}\sum_{n=1}^M\left(f_n(\bar{\bx}_n)-f_n({\bx}^\sharp)\right)
		\le\frac{1}{M}\sum_{n=1}^M (\bar{\bx}_n-{\bx}^\sharp)\cdot\nabla f_n(\bar{\bx}_n) =\frac{1}{M} \sum_{n=1}^M\frac{{\br}_n\cdot\left({\bx}^\sharp -\bar{\bx}_n\right)}{1+{\br}_n\cdot \bar{\bx}_n}.
	\end{align}
	On the other hand, note that
	\[  {\bar {\bx}}_n = ({\bar x}_n^1, \cdots, {\bar x}_n^d), \quad {\bx}^\sharp = (x^{\sharp,1}, \cdots, x^{\sharp,d}), \quad  \sum_{l=1}^{d} {\bar x}^l_n = 1, \quad \sum_{l=1}^{d} x^{\sharp,l} = 1 \]
	and denote
	\[{\bf j}_n:=1+{{\br}_n},\quad\forall~n\ge0,\]
	to see
	\begin{align}
		\begin{aligned} \label{C-7}
			& 1 = {\bf 1} \cdot \bar{\bx}_n, \quad 1+{\br}_n\cdot \bar{\bx}_n = ({\bf 1} + {\br}_n ) \cdot \bar{\bx}_n ={\bf j}_n \cdot \bar{\bx}_n, \\
			& {\br}_n\cdot\big({\bx}^\sharp-\bar{\bx}_n\big) = ({\bf 1} + {\br}_n ) \cdot ({\bx}^\sharp-\bar{\bx}_n) - {\bf 1} \cdot   ({\bx}^\sharp-\bar{\bx}_n) =  {\bf j}_n \cdot  ({\bx}^\sharp-\bar{\bx}_n).
		\end{aligned}
	\end{align}
	Note that the bounded assumption \eqref{r_bound} for ${\br}_n$ can be restated as
	\begin{align*}
		j_n^i\in(\underline{j},\overline{j}),\quad\forall~i\in[d],~\forall~n\ge0,
	\end{align*}
	for some constants $0<\underline{j}<1<\overline{j}$. Finally, we combine \eqref{C-6} and \eqref{C-7} to obtain
	\begin{align} \label{C-8}
		R_M \le \frac{1}{M}\sum_{n=1}^M\frac{{\bf j}_n\cdot\left({\bx}^\sharp-\bar{\bx}_n\right)}{{\bf j}_n\cdot \bar{\bx}_n}
		=\frac{1}{M}\sum_{n=1}^M\frac{ {\bf j}_n\cdot {\bx}^\sharp}{ {\bf j}_n\cdot\bar{\bx}_n}-1.
	\end{align}
	We use this to obtain the expectation of upper bound for $R_M$ as follows:
	\begin{align*}
		R_M&\le\frac{1}{M}\sum_{n=1}^M\frac{{\bf j}_n\cdot {\bx}^\sharp}{ {\bf j}_n\cdot\bar{\bx}_n}-1\le\frac{1}{M}\sum_{n=1}^M\frac{{\bf j}_n\cdot {\bx}_{n}^\sharp}{{\bf j}_n\cdot\bar{\bx}_n}-1\\
		&=\frac{1}{M}\sum_{n=1}^M\Bigg[\left(\frac{ {\bf j}_n\cdot {\bx}_{n}^\sharp}{{\bf j}_n\cdot\bar{\bx}_n}-\frac{{\bf j}_{n-1}\cdot {\bx}_{n-1}^\sharp}{{\bf j}_n\cdot\bar{\bx}_n}\right)+\left(\frac{ {\bf j}_{n-1}\cdot {\bx}_{n-1}^\sharp}{ {\bf j}_n\cdot\bar{\bx}_n}-\frac{ {\bf j}_{n-1}\cdot {\bx}_{n-1}^\sharp}{{\bf j}_{n-1}\cdot\bar{\bx}_n}\right) \\
		&\hspace{2.2cm} +\left(\frac{{\bf j}_{n-1}\cdot {\bx}_{n-1}^\sharp}{{\bf j}_{n-1}\cdot\bar{\bx}_n}-\frac{ {\bf j}_{n-1}\cdot\bar{\bx}_{n}}{ {\bf j}_{n-1}\cdot\bar{\bx}_n}\right)\Bigg]\\
		&=:\frac{1}{M}\sum_{n=1}^M \Big [ {\mathcal I}_{11} + {\mathcal I}_{12} + {\mathcal I}_{13} \Big ].
	\end{align*}
	This implies
	\begin{align} \label{C-10-1}
		\mathbb{E}_{\bx}\Big[ \mathbb{E}_{\bf j}\big[R_M\big] \Big] \leq \frac{1}{M}\sum_{n=1}^M \left(\mathbb{E}_{\bx}\Big[ \mathbb{E}_{\bf j}\big[\mathcal I_{11}\big] \Big]+\mathbb{E}_{\bx}\Big[ \mathbb{E}_{\bf j}\big[\mathcal I_{12}\big] \Big]+\mathbb{E}_{\bx}\Big[ \mathbb{E}_{\bf j}\big[\mathcal I_{13}\big] \Big]\right).
	\end{align}
	In what follows, we provide estimates for $\mathcal I_{11},~\mathcal I_{12}$ and $\mathcal I_{13}$ one by one. \newline
	
	\noindent $\bullet$~(Estimate of $\mathbb{E}_{\bx}\mathbb{E}_{\bf j}\mathcal I_{11}$): Since ${\bf j}_n$ and ${\bf j}_{n-1}$ are independent for each $n\in\mathbb N$, and ${\bx}_n^\sharp$ only depends on ${\bf j}_n$, we have
	\begin{align}\label{C-10-2}
		\begin{aligned}
			&\mathbb E_{\bf j}\left[\Big( {\bf j}_n\cdot {\bx}_{n}^\sharp- {\bf j}_{n-1}\cdot {\bx}_{n-1}^\sharp \Big)^2\right]\\
			&\hspace{1cm}=\mathbb E_{\bf j}\left[\left({\bf j}_n\cdot{\bx}_n^\sharp\right)^2\right]+\mathbb E_{\bf j}\left[\left({\bf j}_{n-1}\cdot{\bx}_{n-1}^\sharp\right)^2\right]-2\mathbb E_{\bf j}\left[{\bf j}_n\cdot{\bx}_n^\sharp\right]\mathbb E_{\bf j}\left[{\bf j}_{n-1}\cdot{\bx}_{n-1}^\sharp\right]\\
			&\hspace{1cm}=2\mathbb E_{\bf j}\left[\left({\bf j}_n\cdot{\bx}_n^\sharp\right)^2\right]-\left(\mathbb E_{\bf j}\left[{\bf j}_n\cdot{\bx}_n^\sharp\right]\right)^2,
		\end{aligned}
	\end{align}
	where we have used
	\[
	\mathbb{E}_{\bf j}  \left[ \Big(  {\bf j}_n\cdot {\bx}_{n}^\sharp \Big)^2 \right]  =  \mathbb{E}_{\bf j}  \left [ \Big( {\bf j}_{n-1}\cdot {\bx}_{n-1}^\sharp  \Big)^2 \right], \quad
	\mathbb{E}_{\bf j}  \Big[ {\bf j}_n\cdot {\bx}_{n}^\sharp \Big]=  \mathbb{E}_{\bf j} \Big[  {\bf j}_{n-1}\cdot {\bx}_{n-1}^\sharp  \Big ].
	\]
	Now, we use the Jensen's inequality and \eqref{C-10-2} to see
	\begin{align*}
		\Big(\mathbb{E}_{\bf j}\big[ {\mathcal I}_{11}\big] \Big)^2
		&\le \mathbb{E}_{\bf j}\Big[ {\mathcal I}_{11}^2 \Big]  =\mathbb{E}_{\bf j} \left[\left(\frac{ {\bf j}_n\cdot {\bx}_{n}^\sharp
			- {\bf j}_{n-1}\cdot {\bx}_{n-1}^\sharp} {{\bf j}_n\cdot\bar{\bx}_n}\right)^2\right] \\
		&\le\frac{2}{\underline{j}^2}\left(\mathbb{E}_{\bf j}\left[\left({\bf j}_n \cdot {\bx}_{n}^\sharp\right)^2\right]-\left(\mathbb{E}_{\bf j} \left[ {\bf j}_n\cdot {\bx}_{n}^\sharp \right]\right)^2\right) =\frac{2}{\underline{j}^2} \mbox{Var}\left( {\bf j}_n\cdot {\bx}_{n}^\sharp\right).
	\end{align*}
	This implies
	\begin{align}\label{C-11}
		\mathbb{E}_{\bf j}\left[ {\mathcal I}_{11}\right]  \leq \frac{\sqrt{2}}{\underline{j}} \sqrt{ \mbox{Var}\left( {\bf j}_n\cdot {\bx}_{n}^\sharp\right)}.
	\end{align}

	\vspace{0.2cm}
	
	\noindent $\bullet$~(Estimate of $\mathbb{E}_{\bx}\mathbb{E}_{\bf j}\mathcal I_{12}$): Similarly, we have
	\begin{align*}
		\begin{aligned}
			{\mathcal I}_{12} &=\frac{{\bf j}_{n-1}\cdot {\bx}_{n-1}^\sharp}{ {\bf j}_n\cdot\bar{\bx}_n}-\frac{{\bf j}_{n-1}\cdot {\bx}_{n-1}^\sharp}{{\bf j}_{n-1}\cdot\bar{\bx}_n}=\frac{{\bf j}_{n-1}\cdot {\bx}_{n-1}^\sharp}{\left({\bf j}_n\cdot\bar{\bx}_n\right)\left({\bf j}_{n-1}\cdot\bar{\bx}_n\right)}\Big(\big({\bf j}_{n-1}- {\bf j}_n\big)\cdot\bar{\bx}_n\Big)\\
			&\le \frac{1}{\underline{j}^2}({\bf j}_{n-1}\cdot {\bx}_{n-1}^\sharp)\cdot| {\bf j}_{n-1}- {\bf j}_n|\le\frac{(\overline{j}-\underline{j})\sqrt{d}}{\underline{j}^2}({\bf j}_{n-1}\cdot {\bx}_{n-1}^\sharp).
		\end{aligned}
	\end{align*}
	This yields
	\begin{align}  \label{C-12}
		\mathbb{E}_{\bf j}\left[{\mathcal I}_{12}\right] \leq  \frac{(\overline{j}-\underline{j})\sqrt{d}}{\underline{j}^2} \mathbb{E}_{\bf j}\Big[ {\bf j}_{n-1}\cdot {\bx}_{n-1}^\sharp \Big] = \frac{(\overline{j}-\underline{j})\sqrt{d}}{\underline{j}^2} \mathbb{E}_{\bf j}\Big[ {\bf j}_{n}\cdot {\bx}_{n}^\sharp \Big].
	\end{align}
	\vspace{0.2cm}
	
	\noindent $\bullet$~(Estimate of $\mathbb{E}_{\bx}\mathbb{E}_{\bf j}\mathcal I_{13}$): Here, we introduce a new index $k_n\in[d]$ which is chosen to satisfy
	\begin{align}\label{beta_x}
		k_n=\argmax_{i\in[N]}\left({\br}_n\cdot\bx_n^i\right),\quad\mbox{hence,}\quad{\bx}_n^{k_n}={\bf M}_{\infty,n}.
	\end{align}
	Note that
	\begin{align}\label{A8-1}
		\begin{aligned}
			{\mathcal I}_{13} &=\frac{ {\bf j}_{n-1}\cdot {\bx}_{n-1}^\sharp}{{\bf j}_{n-1}\cdot\bar{\bx}_n}-\frac{{\bf j}_{n-1}\cdot\bar{\bx}_{n}}{{\bf j}_{n-1}\cdot\bar{\bx}_n}=\frac{{\bf j}_{n-1}\cdot({\bx}_{n-1}^\sharp-\bar{\bx}_{n})}{{\bf j}_{n-1}\cdot\bar{\bx}_n}\\
			&=\frac{{\bf j}_{n-1}\cdot\left({\bx}_{n-1}^\sharp-{\bx}_n^{k_{n-1}}+{\bx}_n^{k_{n-1}}-\bar{\bx}_n\right)}{{\bf j}_{n-1}\cdot\bar{\bx}_n}.
		\end{aligned}
	\end{align}
	By \eqref{beta_x}, we have
	\begin{align*}
		{\bx}_n^{k_{n-1}}&={\bx}_{n-1}^{k_{n-1}}-\lambda_0h({\bx}_{n-1}^{k_{n-1}}-{\bf M}_{\infty,n-1})-\lambda_1h(\bar{\bx}_{n-1}-{\bf M}_{\infty,n-1})+\sigma({\bx}_{n-1}^{k_{n-1}}-M_{\infty,n-1})\odot{\bf W}_{h,n-1}\\
		&={\bf M}_{\infty,n-1}-\lambda_1h(\bar{\bx}_{n-1}-{\bf M}_{\infty,n-1}).
	\end{align*}
	Now, we subtract the following relation
	\begin{align*}
		\bar{\bx}_n=\bar{\bx}_{n-1}-h\left(\lambda_0+\lambda_1\right)\left(\bar{\bx}_{n-1}-{\bf M}_{\infty,n-1}\right)+\sigma\left(\bar{\bx}_{n-1}-{\bf M}_{\infty,n-1}\right)\odot {\bf W}_{h,n-1},
	\end{align*}
	from \eqref{beta_x} to get
	\begin{align}\label{A8-2}
		{\bx}_n^{k_{n-1}}-\bar{\bx}_n=\left(1-\lambda_0h\right)\left({\bf M}_{\infty,n-1}-\bar{\bx}_n\right)-\sigma\left(\bar{\bx}_{n-1}-{\bf M}_{\infty,n-1}\right)\odot {\bf W}_{h,n-1}.
	\end{align}
	Note that
	\[{\bf j}_{n-1}\cdot({\bx}_{n-1}^\sharp-{\bx}_n^{k_{n-1}}),\quad{\bf j}_{n-1}\cdot(\bx_{n-1}^{k_{n-1}}-\bar{\bx}_n)>0,\]
	hence, we use \eqref{A8-1} and \eqref{A8-2} to find
	\begin{align}\label{A8}
		\begin{aligned}
			\mathbb E_{\bf j}\big[\mathcal I_{13}\big]\le&\frac{1}{\underline{j}}
			\left(\mathbb E_{\bf j}\left[{\bf j}_{n-1}\cdot({\bx}_{n-1}^\sharp-{\bx}_n^{k_{n-1}})\right]+(1-\lambda_0h)\mathbb E_{\bf j}\left[{\bf j}_{n-1}\cdot(\bx_{n-1}^{k_{n-1}}-\bar{\bx}_n)\right]\right)\\
			&-\mathbb E_{\bf j}\left[\frac{{\bf j}_{n-1}\cdot\big(\sigma(\bar{\bx}_{n-1}-{\bf M}_{\infty,n-1})\odot {\bf W}_{h,n-1}\big)}{{\bf j}_{n-1}\cdot\bar{\bx}_n}\right].
		\end{aligned}
	\end{align}
	We combine all the estimates \eqref{C-10-1}, \eqref{C-11}, \eqref{C-12} and \eqref{A8} to derive
	\begin{align}\label{rb}
		\begin{aligned}
			&\mathbb E_{\bx}\left[\mathbb E_{\bf j}\left[R_M\right]\right]\\
			&\hspace{.5cm}\le\frac{1}{M\underline{j}}\sum_{n=1}^M\Bigg(\sqrt{2\mbox{Var}\left({\bf j}_n\cdot\bx_n^\sharp\right)}+\frac{(\overline{j}-\underline{j})\sqrt{d}}{\underline{j}}\mathbb E_{\bf j}\left[{\bf j}_n\cdot\bx_n^\sharp\right]\\
			&\hspace{1.5cm}+\mathbb{E}_{\bx} \left[\mathbb{E}_{\bf j}\left[ {\bf j}_{n-1}\cdot ({\bx}_{n-1}^\sharp-{\bx}_n^{k_{n-1}} ) \right]\right]+(1-\lambda_0h)\mathbb E_{\bx}\left[\mathbb E_{\bf j}\left[{\bf j}_{n-1}\cdot(\bx_n^{k_n-1}-\bar\bx_n)\right]\right]\Bigg).
		\end{aligned}
	\end{align}
	Note that the upper bound given in \eqref{rb} has four terms. First two are about the market itself, and the last two can be estimated according to the consensus and optimized value.
	
\end{document}